\documentclass[reqno]{amsart}
\usepackage{amsmath}
\usepackage{graphicx, epsfig}
\usepackage{amssymb}
\usepackage{amsfonts}

\newtheorem{theorem}{Theorem}
\theoremstyle{plain}

\newtheorem{corollary}{Corollary}

\newtheorem{remark}{Remark}

\begin{document}
\title[Short Title]{On the affine representations of the trefoil knot group}
\author[H.Hilden]{Hugh M. Hilden}
\address[H.Hilden]{ Department of Mathematics, University of Hawaii,
Honolulu, HI 96822, USA}
\author[M.T.Lozano]{ Mar\'{\i}a Teresa Lozano*}
\address[M.T.Lozano]{ IUMA, Departamento de Matem\'{a}ticas, Universidad de
Zaragoza, Zaragoza 50009, Spain }
\thanks{*This research was supported by grant MTM2007-67908-C02-01}
\author[J.M.Montesinos]{ Jos\'{e} Mar\'{\i}a Montesinos-Amilibia**}
\address[J.M.Montesinos]{Departamento de Geometr\'{\i}a y Topolog\'{\i}a,
Universidad Complutense, Madrid 28040, Spain}
\thanks{**This research was supported by grant MTM2006-00825}
\date{Mars, 2010}
\dedicatory{}

\begin{abstract}
 The complete classification of representations of the Trefoil knot group $G$ in $S^{3}$ and $SL(2,\mathbb{R})$, their affine deformations, and some
geometric interpretations of the results, are given. Among other results, we also obtain
the classification up to conjugacy of the non cyclic groups of affine
Euclidean isometries generated by two isometries $\mu $ and $\nu $ such that
$\mu ^{2}=\nu ^{3}=1$ , in particular those which
are crystallographic. We also prove that
there are no affine crystallographic groups in the three dimensional
Minkowski space which are quotients of $G$.
\end{abstract}

\maketitle
\section{ Introduction}

The representation of a knot group in the group of isometries of a geometric
manifold is important in order to obtain invariants of the knot and also in
order to relate geometric structures with the knot.

In \cite{HLM2009} the varieties $V(\mathcal{I}_{G}^{c})$ and  $V(\mathcal{I}_{aG}^{c})$
 of c-representations and affine c-representations (resp.) of a
two-generator group in a quaternion algebra are defined. A c-representation is a representation where the image of the generators are conjugate elements. We gave there the
c-representation associated to each point in the varieties and also the
complete classification of c-representations of $G$ in $S^{3}$ and $SL(2,\mathbb{R})$. 

In this article we apply the results of \cite{HLM2009} to the group $G$ of
the Trefoil knot, giving the complete classification of representations of $G$ in $S^{3}$ and $SL(2,\mathbb{R})$. We classify their affine deformations and we give some
geometric interpretations of the results.
We also obtain as a consequence
the classification up to conjugacy of the non cyclic groups of affine
Euclidean isometries which are quotients of $G$, indeed those which are generated by two isometries $\mu $ and $\nu $ such that
$\mu ^{2}=\nu ^{3}=1$ (Theorem \ref{tposiblesH}), in particular those which
are crystallographic (Theorem \ref{tposibleseucli}). We also prove that
there are no affine crystallographic groups in the three dimensional
Minkowski space which are quotients of $G$, by using Mess's theorem (\cite{Mess1990} ,\cite{GM2000}) and
Margulis invariant (\cite{Mar1983}, \cite{Mar1984}).

This paper is organized as follows. In Sec. 2 we recall some concepts  contained in \cite{HLM2009}, about quaternion algebras, the varieties $V(\mathcal{I}_{G}^{c})$ and  $V(\mathcal{I}_{aG}^{c})$.
 of c-representations and affine c-representations of a group $G$ in a quaternion algebra. We include as Theorem \ref{teorema4} the result in \cite{HLM2009} giving explicitly the complete classification of c-representations of $G$ in $S^{3}$ and $SL(2,\mathbb{R})$. In Sec. 3 we  apply Theorem \ref{teorema4} to the Trefoil knot group $G(3_{1})$ and we describe the five occurring cases of representations  associated to the real points of the algebraic variety $V(\mathcal{I}_{G(3_{1})})$. We give also a geometric interpretation of the image of the  representation in each of these five cases as the holonomy of a 2-dimensional geometric cone-manifold.  The geometry of these cone-manifolds is spherical, Euclidean or hyperbolic. In Sec. 4 we obtain all the representations of the Trefoil knot group $G(3_{1})$ in the affine isometry group $A(H)$ of a quaternion algebra $H$. In particular we obtain the representations in the 3-dimensional affine Euclidean (and Lorentz) isometry group. Finally we study the Euclidean and Lorentz crystallographic groups which are images of representations of $G(3_{1})$ and we deduce some interesting consequences.

\section{Preliminaires}

\subsection{Quaternion algebras}

Recall that the quaternion algebra $H=\left( \frac{\mu ,\nu }{k}\right) $ is
the $k$-algebra on two generators $i,j$ with the defining relations:
\begin{equation*}
i^{2}=\mu ,\qquad j^{2}=\nu \qquad \text{and}\qquad ij=-ji.
\end{equation*}
Then $H$ also is a four dimensional vector space over $k$, with basis $\left\{ 1,i,j,ij\right\}. $Given a quaternion $A=\alpha +\beta i+\gamma
j+\delta ij$, $A\in H=\left( \frac{\mu ,\nu }{k}\right) =\langle 1,\quad
i,\quad j,\quad ij\rangle $, we use the notation $A^{+}=\alpha $, and $A^{-}=\beta i+\gamma j+\delta ij$. Then, $A=A^{+}+A^{-}$.

The \emph{conjugate} of $A$, is by definition, the quaternion $\overline{A}:=A^{+}-A^{-}=\alpha -\beta i-\gamma j-\delta ij.$

The \emph{trace} of $A$ is $T(A)=A+\overline{A}$. Therefore $T(A)=2A^{+}=2\alpha \in k.$

The \emph{norm} of $A$ is $N(A):=A\overline{A}=\overline{A}A$. Observe that
the norm can be considered as a quadratic form on $H:$
\begin{equation*}
N(A)=(\alpha +\beta i+\gamma j+\delta ij)(\alpha -\beta i-\gamma j-\delta
ij)=\alpha ^{2}-\beta ^{2}\mu -\gamma ^{2}\nu +\delta ^{2}\mu \nu \in k.
\end{equation*}

We denote by $(H,N)$ the quadratic structure in $H$ defined by the norm.

In the quaternion algebra $M(2,k\mathbb{)=}\left( \frac{-1,1}{k}\right) $
the trace is the usual trace of the matrix, and the norm is the determinant
of the matrix.

There are two important subsets in the quaternion algebra $H=\left( \frac{\mu ,\nu }{k}\right) $: the pure quaternions $H_{0}=\left\{ A\in
H:A^{+}=0\right\} $ (a 3-dimensional vector space over $k$ generated by $\left\{ i,j,ij\right\} $), and the unit quaternions $U_{1}$ (the
multiplicative group of quaternions with norm 1). The norm induces a
quadratic structure on the pure quaternions $(H_{0},N).$

There exists a homomorphism $c:U_{1}\longrightarrow SO(H_{0},N)$ such that $c(A)$ acts on $H_{0}$ by conjugation: $c(A)(B^{-})=AB^{-}\overline{A}$. This
homomorphism permits us to associate to each representation $\rho
:G\longrightarrow U_{1}$ a linear isometry $c\circ \rho $ of the metric
space $(H_{0},N).$

The \emph{equiform group} or \emph{group of similarities} $\mathcal{E}q(H)$
of a quaternion algebra $H$, is the semidirect product $H_{0}\rtimes U$,
where $U$ is the multiplicative group of invertible elements in $H.$ This is
the group whose underlying space is $H_{0}\times U$ and the product is
\begin{equation*}
\begin{array}{lll}
\mathcal{E}q(H)\times \mathcal{E}q(H) & \longrightarrow & \mathcal{E}q(H) \\
((v,A),(w,B)) & \rightarrow & (v+c(A)(w),AB)
\end{array}
\end{equation*}

The group of \emph{affine isometries} $\mathcal{A}(H)$ of a quaternion
algebra is the subgroup of $\mathcal{E}q(H)\ $which is the semidirect
product $H_{0}\rtimes U_{1}$.

The group $\mathcal{A}(H)$ defines a left action on the 3-dimensional vector
space $H_{0}$,
\begin{equation*}
\begin{array}{llll}
\Psi : & \mathcal{A}(H_{0})\times H_{0} & \longrightarrow & H_{0} \\
& ((v,A),u) & \rightarrow & (v,A)u:=v+c(A)(u)%
\end{array}%
\end{equation*}

For an element $(v,A)\in \mathcal{E}q(H)$, $A$ is the \emph{linear part} of $(v,A)$, $N(A)$ is the \emph{homothetic factor}, and $v$ is the \emph{translational part}. For an element $(v,A)\in \mathcal{A}(H)$, the vector $v\in H_{0}$ can be decomposed in a unique way as the orthogonal sum of two
vectors, one of them in the $A^{-}$ direction:
\begin{equation*}
v=sA^{-}+v^{\perp },\quad \qquad \left\langle v^{\perp },A^{-}\right\rangle
=0
\end{equation*}
Then
\begin{equation*}
(v,A)=(v^{\perp },1)(sA^{-},A)
\end{equation*}
The element $(v^{\perp },1)$ is a translation in $H_{0}$. The restriction of
the action of $(sA^{-},A)$ on the line generated by $A^{-}$ is a translation
with vector $sA^{-}$. We define $sA^{-}$ as the \emph{vector shift of the
element }$(v,A)$. The length $\sigma $ of the vector shift will be called
the \emph{shift of the element }$(A,v).$

The action of $(v,A)\ $leaves (globally) invariant an affine line parallel
to $A^{-}$ and its action on this line is a translation with vector the
shift $sA^{-}$. This invariant affine line will be call \emph{the axis of} $(v,A)$ . Then the action of $(v,A)$ on the axis of $(v,A)$ is a translation
by $\sigma $.

The following Table contains three important examples, which are the only
quaternion algebras over $\mathbb{R}$ and $\mathbb{C}$ up to isomorphism.

\begin{tabular}{|c|c|c|c|}
\hline
$H=\left( \frac{\mu ,\nu }{k}\right) $ & $U_{1}$ & $H_{0}$ & $(H_{0},N)$ \\
\hline \hline
$M(2,\mathbb{C})=\left( \frac{-1,1}{\mathbb{C}}\right) $ & $SL(2,\mathbb{C)}$
& $\mathbb{C}^{3}$ & $\mathbb{C}^{3}$ Complex Euclidean space \\ \hline
$M(2,\mathbb{R})=\left( \frac{-1,1}{\mathbb{R}}\right) $ & $SL(2,\mathbb{R)}$
& $\mathbb{R}^{3}$ & $E^{2,1}$ Minkowski space \\ \hline
$\mathbb{H}=\left( \frac{-1,-1}{\mathbb{R}}\right) $ & $S^{3}=SU(2,\mathbb{R})$ & $\mathbb{R}^{3}$ & $E^{3}$ (Real) Euclidean space \\ \hline
\end{tabular}

In fact, $M(2,\mathbb{R})$ and the Hamilton quaternions $\mathbb{H}=\left(
\frac{-1,-1}{\mathbb{R}}\right) $ are each isomorphic to an $\mathbb{R}$-subalgebra of $M(2,\mathbb{C}).$

\subsection{The varieties $V(\mathcal{I}_{G}^{c})$ and $V(\mathcal{I}_{aG}^{c})$}

Let $G$ be a group given by the presentation
\begin{equation*}
G=\left\vert a,b:w(a,b)\right\vert
\end{equation*}
where $w$ is a word in $a$ and $b$. A homomorphism
\begin{equation*}
\rho :G\longrightarrow U_{1}
\end{equation*}%
such that $\rho (a)=A$ and $\rho (b)=B$ are conjugate elements in $U_{1}$ is
called here a \textit{c-representation}. Then, since $A$ and $B$ are
conjugate quaternions, $A^{+}=B^{+}$. Set
\begin{equation*}
x=A^{+}=B^{+}\quad \text{and \quad }y=-(A^{-}B^{-})^{+}.
\end{equation*}
We say that $\rho :G\longrightarrow U_{1}$ realizes $(x,y).$ It is proven in
\cite{HLM2009} that if $U_{1}$ is the group of unit quaternions in $M(2,\mathbb{C})=\left( \frac{-1,1}{\mathbb{C}}\right) ,$ then there is an
algorithm to construct an ideal $\mathcal{I}_{G}^{c}$ generated by four
polynomials
\begin{equation*}
\left\{ p_{1}(x,y),p_{2}(x,y),p_{3}(x,y),p_{4}(x,y)\right\}
\end{equation*}
with integer coefficients defining the \emph{algebraic variety of
c-representations} of $G$: $V(\mathcal{I}_{G}^{c})$. It is characterized as
follows:
\begin{enumerate}
\item [-]The set of points $\{(x,y)\in V(\mathcal{I}_{G}^{c}):y^{2}\neq
(1-x^{2})^{2}\}$ coincides with the pairs $(x,y)$ for which there exists an
irreducible c-representation $\rho :G\longrightarrow U_{1}$, unique up to
conjugation in $U_{1}$, realizing $(x,y)$.

\item [-]The set of points $\{(x,y)\in V(\mathcal{I}_{G}^{c}):y^{2}=(1-x^{2})^{2},\quad x^{2}\neq 1\}$ coincides with the pairs $(x,y)$ for which there exists an almost- irreducible c-representation $\rho
:G\longrightarrow U_{1}$, unique up to conjugation in $U_{1}$, realizing $(x,y)$.

\item [-]The set of points $\{(x,y)\in V(\mathcal{I}_{G}^{c}):y=0,\quad
x^{2}=1\}$ coincides with the pairs $(x,y)$ for which neither irreducible
nor almost-irreducible c-representations $\rho :G\longrightarrow U_{1}$
realizing $(x,y)$ exist.
\end{enumerate}

The real points in $V(\mathcal{I}_{G}^{c})$ , excepting the cases $\{(x,\pm
(1-x^{2}))|x^{2}\leq 1\}$, correspond to irreducible c-representations in $S^{3}$ and irreducible (or almost-irreducible) c-representations in $SL(2,\mathbb{R})$ according to \cite[Th.4]{HLM2009}:

\begin{theorem}[\cite{HLM2009}]
\label{teorema4}Let $G$ be a group given by the presentation%
\begin{equation*}
G=\left\vert a,b:w(a,b)\right\vert
\end{equation*}
where $w$ is a word in $a$ and $b$. If $(x_{0}$, $y_{0})$ is a real point of
the algebraic variety $V(\mathcal{I}_{G}^{c})$ we distinguish two cases:

\begin{enumerate}
\item If%
\begin{equation*}
1-x_{0}^{2}>0,\quad (1-x_{0}^{2})^{2}>y_{0}^{2}
\end{equation*}
there exists an irreducible c-representation $\rho :G\longrightarrow U_{1}$,
$U_{1}=S^{3}\subset \mathbb{H}$,~unique up to conjugation in $U_{1}=S^{3}$,
realizing $(x_{0},y_{0})$. Namely:
\begin{equation*}
\fbox{$
\begin{array}{l}
A=x_{0}+\frac{1}{+\sqrt{1-x_{0}^{2}}}\left( y_{0}i+\sqrt{(1-x_{0}^{2})^{2}-y_{0}^{2}}j\right)  \\
B=x_{0}+\sqrt{1-x_{0}^{2}}i,\quad \sqrt{1-x_{0}^{2}}>0
\end{array}
$}
\end{equation*}
\begin{equation*}
(A^{-}B^{-})^{-}=-\sqrt{(1-x^{2})^{2}-y^{2}}ij
\end{equation*}

\item The remaining cases. Then excepting the case
\begin{equation*}
1-x_{0}^{2}>0,\quad (1-x_{0}^{2})^{2}=y_{0}^{2}
\end{equation*}
and the case
\begin{equation*}
x_{0}^{2}=1,\quad y_{0}=0
\end{equation*}
there exists an irreducible (or almost-irreducible) c-representation $\rho
:G\longrightarrow U_{1}$, $U_{1}=SL(2,\mathbb{R})\subset \left( \frac{-1,1}{\mathbb{R}}\right) $ realizing $(x_{0},y_{0})$. Moreover two such
homomorphisms are equal up to conjugation in the group of quaternions with
norm $\pm 1$, $U_{\pm 1}$. Specifically:

\item[(2.1)] If
\begin{equation*}
1-x_{0}^{2}>0,\quad (1-x_{0}^{2})^{2}<y_{0}^{2}
\end{equation*}
set
\begin{equation*}
\fbox{$
\begin{array}{c}
A=x_{0}+\sqrt{1-x_{0}^{2}}I,\quad \sqrt{1-x_{0}^{2}}>0 \\
B=x_{0}+\frac{1}{+\sqrt{1-x_{0}^{2}}}\left( y_{0}I+\sqrt{
y_{0}^{2}-(1-x_{0}^{2})^{2}}J\right)
\end{array}
$}
\end{equation*}
Then $\rho :G\longrightarrow U_{1}$ is irreducible, and
\begin{equation*}
(A^{-}B^{-})^{-}=+\sqrt{y_{0}^{2}-(1-x_{0}^{2})^{2}}IJ
\end{equation*}

\item[(2.2)] If
\begin{equation*}
1-x_{0}^{2}<0,\quad (1-x_{0}^{2})^{2}<y_{0}^{2}
\end{equation*}
there are two subcases:

\begin{enumerate}
\item[(2.2.1)] $y_{0}<0$. Set
\begin{equation*}
\fbox{$
\begin{array}{c}
A=x_{0}+\sqrt{x_{0}^{2}-1}J,\quad \sqrt{x_{0}^{2}-1}>0 \\
B=x_{0}-\frac{1}{+\sqrt{x_{0}^{2}-1}}\left( \sqrt{y_{0}^{2}-(x_{0}^{2}-1)^{2}
}I+y_{0}J\right)
\end{array}
$}
\end{equation*}
\begin{equation*}
(A^{-}B^{-})^{-}=-\sqrt{y_{0}^{2}-(x_{0}^{2}-1)^{2}}IJ
\end{equation*}

\item[(2.2.2)] $y_{0}>0$. Set
\begin{equation*}
\fbox{$
\begin{array}{c}
A=x_{0}+\sqrt{x_{0}^{2}-1}J,\quad \sqrt{x_{0}^{2}-1}>0 \\
B=x_{0}+\frac{1}{+\sqrt{x_{0}^{2}-1}}\left( \sqrt{y_{0}^{2}-(x_{0}^{2}-1)^{2}
}I-y_{0}J\right)
\end{array}
$}
\end{equation*}
\begin{equation*}
(A^{-}B^{-})^{-}=+\sqrt{y_{0}^{2}-(x_{0}^{2}-1)^{2}}IJ
\end{equation*}
In both subcases $\rho :G\longrightarrow U_{1}$ is irreducible.
\end{enumerate}

\item[(2.3)] If
\begin{equation*}
1-x_{0}^{2}<0,\quad (1-x_{0}^{2})^{2}>y_{0}^{2}
\end{equation*}
set
\begin{equation*}
\fbox{$
\begin{array}{c}
A=x_{0}+\sqrt{x_{0}^{2}-1}J,\quad \sqrt{x_{0}^{2}-1}>0 \\
B=x_{0}+\frac{1}{+\sqrt{x_{0}^{2}-1}}\left( -y_{0}J+\sqrt{
(x_{0}^{2}-1)^{2}-y_{0}^{2}}IJ\right)
\end{array}
$}
\end{equation*}
Then $\rho :G\longrightarrow U_{1}$ is irreducible and
\begin{equation*}
(A^{-}B^{-})^{-}=-\sqrt{(x_{0}^{2}-1)^{2}-y_{0}^{2}}IJ
\end{equation*}

\item[(2.4)] If
\begin{equation*}
1-x_{0}^{2}<0,\quad (1-x_{0}^{2})^{2}=y_{0}^{2}
\end{equation*}
set%
\begin{equation*}
\fbox{$
\begin{array}{c}
A=x_{0}+\sqrt{x_{0}^{2}-1}J,\quad \sqrt{x_{0}^{2}-1}>0 \\
B=x_{0}+(IJ+I)-\frac{_{y_{0}}}{\sqrt{x_{0}^{2}-1}}J
\end{array}
$}
\end{equation*}
Then $\rho :G\longrightarrow U_{1}$ is almost-irreducible, and
\begin{equation*}
(A^{-}B^{-})^{-}=-\sqrt{(x_{0}{}^{2}-1)}(I+IJ).
\end{equation*}

\item[(2.5)] If
\begin{equation*}
1-x_{0}^{2}=0,\quad y_{0}\neq 0
\end{equation*}
set
\begin{equation*}
\fbox{$
\begin{array}{c}
A=x_{0}+I+J \\
B=x_{0}+\frac{y_{0}}{2}(I-J)
\end{array}
$}
\end{equation*}
Then $\rho :G\longrightarrow U_{1}$ is irreducible, and
\begin{equation*}
(A^{-}B^{-})^{-}=-y_{0}IJ
\end{equation*}
\end{enumerate}
\qed
\end{theorem}

We also want to study the \emph{c-representations of }$G$\emph{\ in the affine
group }$\mathcal{A}(H)$ of a quaternion algebra $H$, that is representations
of $G$ in the affine group $\mathcal{A}(H)$ of a quaternion algebra $H$ such
that the generators $a$ and $b$ go to conjugate elements, up to conjugation
in $\mathcal{E}q(H)$. Given a such c-representation $\rho :G\longrightarrow
\mathcal{A}(H)$, if $\rho (a)$, $\rho (b)$ has translational part different
from $0$, we can suppose that
\begin{equation*}
\begin{array}{llll}
\rho : & G & \longrightarrow & \mathcal{A}(H) \\
& a & \rightarrow & (sA^{-},A) \\
& b & \rightarrow & (sB^{-}+(A^{-}B^{-})^{-},B)
\end{array}
\end{equation*}
where $(A,B)$ is an irreducible pair of conjugate unit quaternions. (A pair $(A,B)$ is  irreducible
if $\{A^{-},B^{-},(A^{-}B^{-})^{-}\}$ is a basis of $H_{0}$.)

Because $\rho $ is a homomorphism of the semidirect product $H_{0}\rtimes
U_{1}=\mathcal{A}(H),$ we have
\begin{equation}
\rho (w(a,b))=(\frac{\partial w}{\partial a}|_{\phi }\circ v+\frac{\partial w
}{\partial b}|_{\phi }\circ u,w(A,B))=(0,I)  \label{erelfox}
\end{equation}
where $\frac{\partial w}{\partial a}|_{\phi }$ is the Fox derivative of the
word $w(a,b)$ with respect to $a$, and evaluated by $\phi $ such that $\phi
(a)=A$, $\phi (b)=B$. (See \cite{CF1963}.) The equation \ref{erelfox} yields
two relations between the parameters
\begin{eqnarray*}
x &=&A^{+}=B^{+} \\
y &=&-(A^{-}B^{-})^{+} \\
s &=&vector\ shift\ parameter
\end{eqnarray*}
the relations are
\begin{equation}
w(A,B)=I  \label{erel1}
\end{equation}
\begin{equation}
\frac{\partial w}{\partial a}|_{\phi }\circ v+\frac{\partial w}{\partial b}
|_{\phi }\circ u=0  \label{erel2}
\end{equation}

The relation \eqref{erel1} yields the ideal $\mathcal{I}_{G}^{c}=
\{p_{i}(x,y)\mid i\in \left\{ 1,2,3,4\right\} \}$, and it defines $V(
\mathcal{I}_{G}^{c})$ the algebraic variety of c-representations of $G$ in $
SL(2,\mathbb{C})$.

The relation \eqref{erel2} produces four polynomials in $x,y,s$: $\{q_{j}(x,y,s)\mid j\in \left\{ 1,2,3,4\right\} \}.$ The ideal
\begin{equation*}
\mathcal{I}_{aG}^{c}=\{p_{i}(x,y),q_{j}(x,y,s)\mid i,j\in \left\{
1,2,3,4\right\} \}
\end{equation*}
defines an algebraic variety, that we call $V_{a}(\mathcal{I}_{aG}^{c})$ the
\emph{variety of affine c-representations of }$G$\emph{\ in }$A(H)$ up to
conjugation in $\mathcal{E}q(H)$.

\subsection{The group of the trefoil knot}

Consider the standard presentation of the group of the trefoil knot $3_{1}$
as the 2-bridge knot $3/1$. Figure \ref{ftrebol}. (See \cite{BZ}.):
\begin{equation*}
G(3_{1})=|a,b;aba=bab|
\end{equation*}
where $a$ and $b$ are meridians. This knot is also the toroidal knot $
\{3,2\} $. As such, it has the following presentation:

\begin{equation}
G(3_{1})=|F,D;F^{3}=D^{2}|  \label{genfd}
\end{equation}
where $F=ab$ and $D=aba$. It is easy to show that the element $
C:=F^{3}=D^{2} $ belongs to the center of $G(3_{1})$.

\begin{figure}[ht]
\epsfig{file=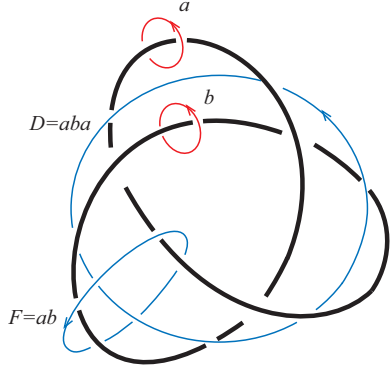,height=5cm} \caption{The trefoil knot $3_{1}$.}\label{ftrebol}
\end{figure}

\section{The representations of $G(3_{1})$ in $U_{1}$}

Let
\begin{equation*}
\begin{array}{cccc}
\rho : & G(3_{1})=|a,b;aba=bab| & \longrightarrow  & U_{1} \\
& a & \rightarrow  & \rho (a)=A \\
& b & \rightarrow  & \rho (b)=B
\end{array}
\end{equation*}
denote a representation of $G=G(3_{1})$ in the unit group of a quaternion
algebra $H$. This always is a c-representation because $a$ and $b$ are
conjugate elements in $G$. Moreover if $\rho $ is irreducible (that is, if
$\{A^{-},B^{-},(A^{-}B^{-})^{-}\}$ is a basis of $H_{0}$, the 3-dimensional
vector space of pure quaternions), then $\{A,B\}$ generates $H$ as an
algebra. Since $\rho (C)$ commutes with $A$ and $B$, it belongs to the
center of $H$. Hence $\rho (C)=\pm 1.$ Applying the algorithm of \cite[Th.3]{HLM2009}
to the presentation
\begin{equation*}
G(3_{1})=|a,b;aba=bab|
\end{equation*}
we obtain $4$ polynomials, two of which are zero and the other two coincide
up to sign with
\begin{equation}
2x^{2}-2y-1.  \label{epoly1}
\end{equation}
Therefore $\mathcal{I}_{G}=(2x^{2}-2y-1)$. The real part of the algebraic
variety $V(\mathcal{I}_{G})$ is the parabola $y=\frac{2x^{2}-1}{2}$ depicted
in Figure \ref{fcurvas}.

\begin{figure}[ht]
\epsfig{file=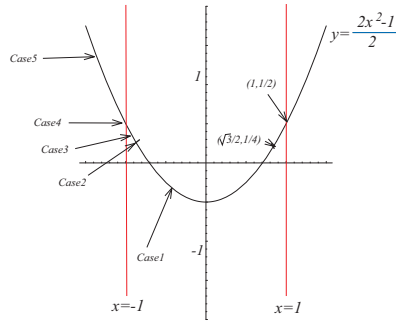,height=4.5cm} \caption{The real curve $V(\mathcal{I}
_{G})$ and the five cases.}\label{fcurvas}
\end{figure}

Theorem \ref{teorema4} establishes the different cases of representations
associated to real points of the algebraic variety $V(\mathcal{I}_{G}).$
Figure \ref{fregiones} shows a pattern on the real plane with coordinates $x$
and $\frac{y}{1-x^{2}}$ : The plane is divided in several labeled regions by
labeled segments. The label corresponds to the case described in Theorem \ref
{teorema4}. Points in the segments between region (1) and (2.3), which have
no label, correspond to almost-irreducible representations in $SL(2,\mathbb{C
})$. Therefore, to apply Theorem \ref{teorema4} to the algebraic variety $V(
\mathcal{I}_{G})$ it suffices to study the graphic of $\frac{y}{1-x^{2}}$ as a
function of $x$ over the pattern.

\begin{figure}[ht]
\epsfig{file=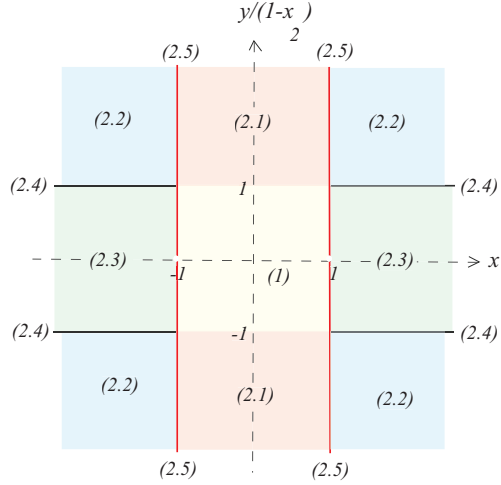,height=6.5cm} \caption{The pattern for different cases.}\label{fregiones}
\end{figure}

Figure \ref{fregiones}
shows $\frac{y}{1-x^{2}}$ as a funtion of $x$ for the algebraic variety $V(\mathcal{I}_{G})$ of the Trefoil knot.

\begin{figure}[ht]
\epsfig{file=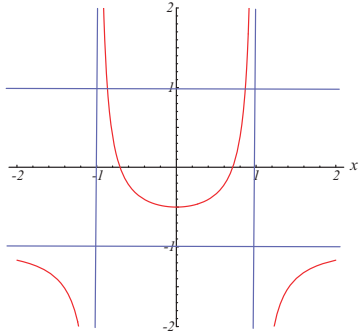,height=4.5cm} \caption{The function $\frac{y}{1-x^{2}}$
for the Trefoil knot.}\label{fysobreuutrebol}
\end{figure}
Then, according with Theorem \ref{teorema4}, there are several
cases:

\subsection{Case 1}
Region (1).
\[x\in (\frac{-\sqrt{3}}{2},\frac{\sqrt{3}}{2}
)\Longleftrightarrow \left\{ 1-x^{2}>0,(1-x^{2})^{2}>y^{2}\right\}, \] where
\begin{equation*}
y=\frac{2x^{2}-1}{2},\quad u=1-x^{2}
\end{equation*}
There exists an irreducible c-representation $\rho
_{x}:G(3_{1})\longrightarrow S^{3}$ realizing $(x,y)$, unique up to
conjugation in $S^{3}$, such that
\begin{equation}
\begin{array}{l}
\rho _{x}(a)=A=x+\frac{2x^{2}-1}{2\sqrt{1-x^{2}}}i+\frac{1}{2}\sqrt{\frac{
3-4x^{2}}{1-x^{2}}}j \\
\rho _{x}(b)=B=x+\sqrt{1-x^{2}}i,\quad \sqrt{1-x^{2}}>0
\end{array}
\label{ecase1}
\end{equation}
The composition of $\rho _{x}$ with $c:S^{3}\rightarrow SO(3)$, where $c(X)$
, $X\in S^{3}$, acts on $P\in H_{0}\cong E^{3}$ by conjugation, defines the
representation $\rho _{x}^{\prime }=c\circ \rho _{x}:G(3_{1})\longrightarrow
SO(3)$. In linear notation, where $\left\{ X,Y,Z\right\} $ is the coordinate
system in $E^{3}$ associated to the basis $\left\{ -ij,j,i\right\} $ we have
\begin{equation*}
\begin{array}{l}
\rho _{x}^{\prime }(a)=m_{x}(a)\left(
\begin{array}{c}
X \\
Y \\
Z%
\end{array}
\right) =\left(
\begin{array}{c}
X^{\prime } \\
Y^{\prime } \\
Z^{\prime }
\end{array}
\right) \\
\rho _{x}^{\prime }(b)=m_{x}(b)\left(
\begin{array}{c}
X \\
Y \\
Z%
\end{array}
\right) =\left(
\begin{array}{c}
X^{\prime } \\
Y^{\prime } \\
Z^{\prime }
\end{array}
\right)
\end{array}
\end{equation*}
where
\begin{equation*}
m_{x}(a)=\left(
\begin{array}{ccc}
2x^{2}-1 & \frac{x-2x^{3}}{\sqrt{1-x^{2}}} & x\sqrt{\frac{4x^{2}-3}{x^{2}-1}}
\\
\frac{x(2x^{2}-1)}{\sqrt{1-x^{2}}} & \frac{-4x^{4}+2x^{2}+1}{2-2x^{2}} &
\frac{(1-2x^{2})\sqrt{3-4x^{2}}}{2(x^{2}-1)} \\
-x\sqrt{\frac{4x^{2}-3}{x^{2}-1}} & \frac{(1-2x^{2})\sqrt{3-4x^{2}}}{
2(x^{2}-1)} & \frac{1-2x^{2}}{2x^{2}-2}
\end{array}
\right)
\end{equation*}
and
\begin{equation*}
m_{x}(b)=
\begin{pmatrix}
2x^{2}-1 & -2x\sqrt{1-x^{2}} & 0 \\
2x\sqrt{1-x^{2}} & 2x^{2}-1 & 0 \\
0 & 0 & 1
\end{pmatrix}
.
\end{equation*}
The maps $\rho _{x}^{\prime }(a)$ and $\rho _{x}^{\prime }(b)$ are right
rotations of angle $\alpha $ around the axes $A^{-}$ and $B^{-}$ where $
x=A^{+}=B^{+}=\cos \frac{\alpha }{2}$. Moreover, $\rho _{x}^{\prime }(C)$ is
the identity. Hence $\rho _{x}^{\prime }:G(3_{1})\longrightarrow SO(3)$
factors through the group $C_{2}\ast C_{3}=|F,D;F^{3}=D^{2}=1|$. Thus $\rho
_{x}^{\prime }(F)=\rho _{x}^{\prime }(ab)$ is a $3$-fold rotation and $\rho
_{x}^{\prime }(D)=\rho _{x}^{\prime }(aba)$ is a $2$-fold rotation.

The angle $\omega $ between the axes of $\rho _{x}^{\prime }(a)$ and $\rho
_{x}^{\prime }(b)$ is given by
\begin{equation*}
\cos \omega =\frac{y}{u}=\frac{x^{2}-\frac{1}{2}}{1-x^{2}}
\end{equation*}
The geometrical meaning of the representation $\rho _{x}^{\prime
}:G(3_{1})\longrightarrow SO(3)$ for $\frac{-\sqrt{3}}{2}<x<\frac{\sqrt{3}}{2
}$ is the following.

\begin{theorem}
\label{tgenso3}The image of $\rho _{x}^{\prime }:G(3_{1})\longrightarrow
SO(3)$ , $-\frac{\sqrt{3}}{2}<x<\frac{\sqrt{3}}{2}$, is isomorphic to the
holonomy of the 2-dimensional spherical cone-manifold $S_{\frac{2\pi }{2},
\frac{2\pi }{3},\alpha }^{2}$, where $x=\cos \frac{\alpha }{2}$ ($\frac{
10\pi }{6}>\alpha >\frac{2\pi }{6}$).

\begin{proof}
The 2-dimensional spherical cone-manifold $S_{\frac{2\pi }{2},\frac{2\pi }{3}
,\alpha }^{2}$ is the result of identifying the edges of the spherical
triangle $T(\frac{2\pi }{3},\frac{\alpha }{2},\frac{\alpha }{2})$ as in the
Figure \ref{s23alfa}. The holonomy of $S_{\frac{2\pi }{2},\frac{2\pi }{3}
,\alpha }^{2}$ is generated by rotations of angle $\alpha $ in the vertices $P$ and $Q$. The distance $r$ between $P$ and $Q$ is calculated by spherical
trigonometry:
\begin{equation*}
\cos r=\frac{\cos ^{2}\frac{\alpha }{2}+\cos \frac{2\pi }{3}}{\sin ^{2}\frac{
\alpha }{2}}=\frac{x^{2}-\frac{1}{2}}{1-x^{2}}=\cos \omega
\end{equation*}
Then $r=\omega $. Therefore the points $P$ and $Q$ are the intersection with
$S^{2}$ of the axes $A^{-}$ and $B^{-}$of the two generators $\rho
_{x}^{\prime }(a)$ and $\rho _{x}^{\prime }(b)$ of the subgroup $\rho
_{x}^{\prime }(G(3/1)).$
\end{proof}
\end{theorem}
\begin{figure}[ht]
\epsfig{file=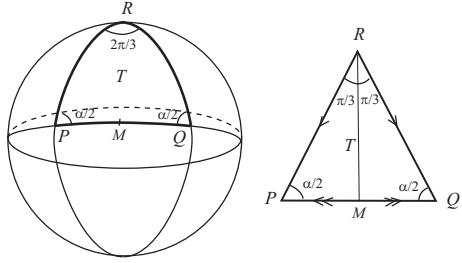,height=3.5cm} \caption{The
spherical conemanifold $S_{\pi ,2\pi /3,\alpha }$.}\label{s23alfa}
\end{figure}
\begin{remark}
The point $R$ (resp. $M$ the mid-point between $P$ and $Q$) is the
intersection with $S^{2}$ of the axis of the $3$-fold (resp. $2$-fold)
rotation $\rho _{x}^{\prime }(F)=\rho _{x}^{\prime }(ab)$ (resp. $\rho
_{x}^{\prime }(D)=\rho _{x}^{\prime }(aba)$).
\end{remark}

\subsection{Case 2} Points in the no-label segments between region (1) and region (2.1).
\[(x,y)=(\pm \frac{\sqrt{3}}{2},\frac{1}{4})\Longleftrightarrow \left\{
1-x^{2}>0,(1-x^{2})^{2}=y^{2}\right\}.\]

There exists an almost-irreducible c-representation $\rho
_{x}:G(3_{1})\longrightarrow U_{1}\subset \left( \frac{-1,1}{\mathbb{C}}
\right) $ realizing $(x,y)$, unique up to conjugation in $U_{1}$, such that:
\begin{eqnarray}
\rho _{\pm \sqrt{3}/2}(a) &=&A=\frac{\pm \sqrt{3}}{2}+\frac{\sqrt{-1}}{2}IJ
\label{ecaso2} \\
\rho _{\pm \sqrt{3}/2}(b) &=&B=\frac{\pm \sqrt{3}}{2}-\frac{1}{2}I+\frac{1}{2
}J+\frac{\sqrt{-1}}{2}IJ  \notag
\end{eqnarray}
This representation cannot be conjugated to any real representation. Under
the isomorphism $U_{1}\approx SL(2,\mathbb{C})$ we have
\begin{equation*}
\begin{array}{cccc}
\rho _{\pm \sqrt{3}/2}: & G(3_{1}) & \longrightarrow & SL(2,\mathbb{C}) \\
& a & \rightarrow & A=
\begin{pmatrix}
\frac{\pm \sqrt{3}}{2}+\frac{\sqrt{-1}}{2} & 0 \\
0 & \frac{\pm \sqrt{3}}{2}-\frac{\sqrt{-1}}{2}
\end{pmatrix}
\\
& b & \rightarrow & B=
\begin{pmatrix}
\frac{\pm \sqrt{3}}{2}+\frac{\sqrt{-1}}{2} & 0 \\
1 & \frac{\pm \sqrt{3}}{2}-\frac{\sqrt{-1}}{2}
\end{pmatrix}
\end{array}
\end{equation*}

The composition of $\rho _{\sqrt{3}/2}$ with $c:U_{1}\rightarrow PSL(2,
\mathbb{C})$, defines the representations $\rho _{\sqrt{3}/2}^{\prime
}=c\circ \rho _{\sqrt{3}/2}:G(3_{1})\longrightarrow PSL(2,\mathbb{C})$ where
(up to conjugation with $w=\frac{-1}{z}$) $\rho _{\sqrt{3}/2}^{\prime }(a)$
is the $\frac{2\pi }{6}$-rotation of $\mathbb{C}P^{1}$around the point $0$:
\begin{equation*}
w=e^{i\frac{2\pi }{6}}z
\end{equation*}
and $\rho _{\sqrt{3}/2}^{\prime }(b)$ is the $\frac{2\pi }{6}$-rotation of $
\mathbb{C}P^{1}$around the point $i$:
\begin{equation*}
w=e^{i\frac{2\pi }{6}}z+e^{i\frac{\pi }{6}}
\end{equation*}

\begin{figure}[ht]
\epsfig{file=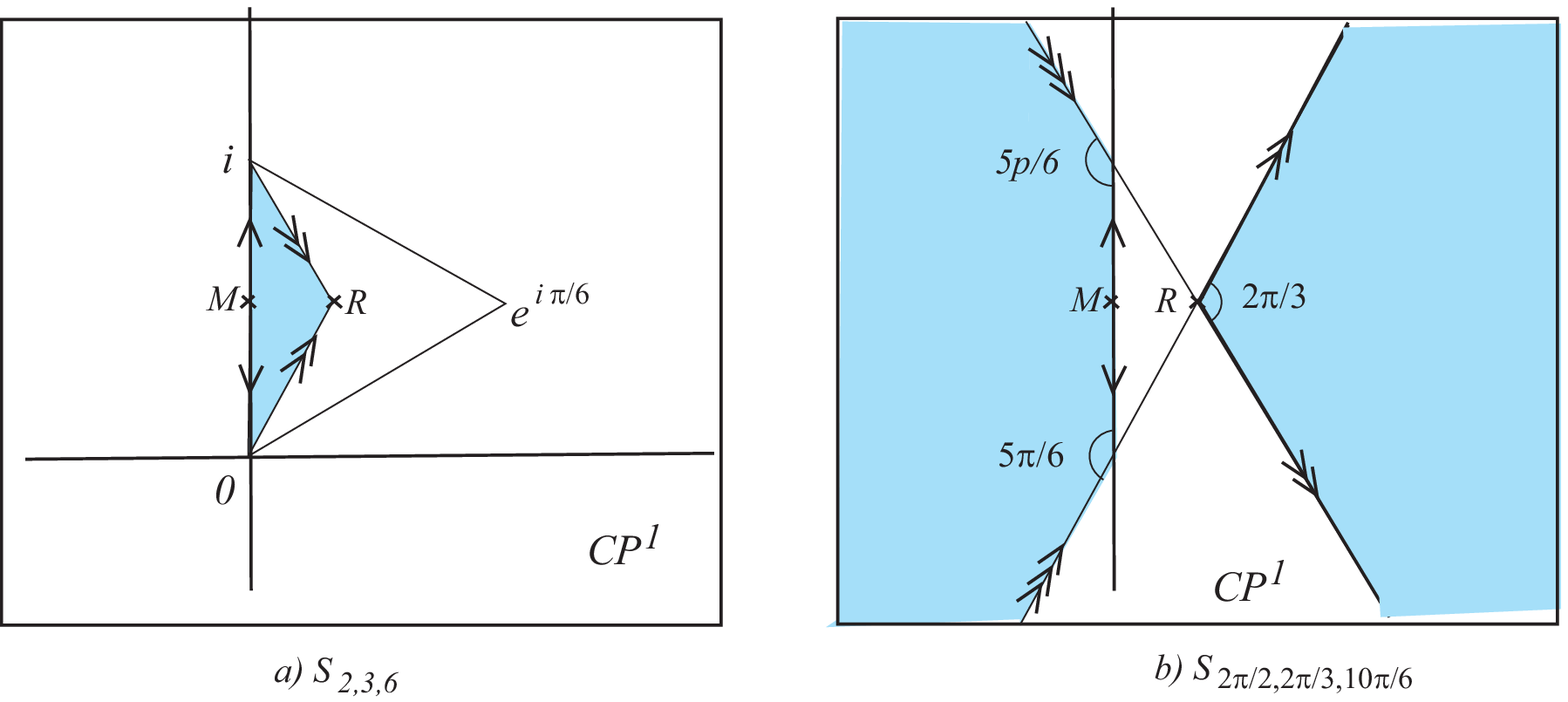,height=4.5cm} \caption{Case 2.}\label{orbifoldS236}
\end{figure}

We see, in fact, that

\begin{theorem}
The image of $\rho _{\sqrt{3}/2}^{\prime }:G(3_{1})\longrightarrow PSL(2,\mathbb{C})$ acts on the Euclidean plane $\mathbb{C}$ as the holonomy of the
Euclidean crystallographic orbifold $S_{2,3,6}^{2}$.
\end{theorem}

\begin{remark}
The barycenter of the triangle $\{0,i,e^{i\frac{\pi }{6}}\}$ is the center
of the $3$-fold rotation $\rho _{x}^{\prime }(F)=\rho _{\sqrt{3}/2}^{\prime
}(ab)$ and the point $\frac{i}{2}$ is the center of the $2$-fold rotation $
\rho _{x}^{\prime }(D)=\rho _{\sqrt{3}/2}^{\prime }(aba)$). See Figure \ref
{orbifoldS236}. a)
\end{remark}

\begin{remark}
The image of $\rho _{\sqrt{3}/2}^{\prime }:G(3_{1})\longrightarrow PSL(2,\mathbb{C})$ can be interpreted as the holonomy of the Euclidean
crystallographic conemanifold $S_{\frac{2\pi }{2},\frac{2\pi }{3},\frac{10\pi }{6}}^{2}$. See Figure \ref{orbifoldS236}. b)
\end{remark}

\subsection{Case 3}
Region (2.1).
\[x\in (-1,\frac{-\sqrt{3}}{2})\cup (\frac{\sqrt{3}}{2}
,1)\Longleftrightarrow \left\{ 1-x^{2}>0,(1-x^{2})^{2}<y^{2}\right\} .\]

There exists an irreducible c-representation $\rho
_{x}:G(3_{1})\longrightarrow SL(2,\mathbb{R})=U_{1}\subset \left( \frac{-1,1
}{\mathbb{R}}\right) $ realizing $(x,y)$, unique up to conjugation in $SL(2,
\mathbb{R})$, such that
\begin{equation*}
\begin{array}{l}
\rho _{x}(a)=A=x+\sqrt{1-x^{2}}i,\quad \sqrt{1-x^{2}}>0 \\
\rho _{x}(b)=B=x+\frac{2x^{2}-1}{2\sqrt{1-x^{2}}}i+\frac{1}{2}\sqrt{\frac{
-3+4x^{2}}{1-x^{2}}}j
\end{array}
\end{equation*}
The composition of $\rho _{x}$ with $c:SL(2,\mathbb{R})\rightarrow
SO^{0}(1,2)\cong Iso^{+}(\mathbb{H}^{2})$, where $c(X)$, $X\in SL(2,\mathbb{R})$, acts on $P\in H_{0}\cong E^{1,2}$ by conjugation, defines the
representation $\rho _{x}^{\prime }=c\circ \rho _{x}:G(3_{1})\longrightarrow
SO^{0}(1,2)$ in affine linear notation, where $\left\{ X,Y,Z\right\} $ is
the coordinate system associated to the basis $\left\{ -ij,j,i\right\} $
\begin{equation*}
\begin{array}{l}
\rho _{x}^{\prime }(a)=m_{x}(a)\left(
\begin{array}{c}
X \\
Y \\
Z%
\end{array}
\right) =\left(
\begin{array}{c}
X^{\prime } \\
Y^{\prime } \\
Z^{\prime }
\end{array}
\right) \\
\rho _{x}^{\prime }(b)=m_{x}(b)\left(
\begin{array}{c}
X \\
Y \\
Z
\end{array}
\right) =\left(
\begin{array}{c}
X^{\prime } \\
Y^{\prime } \\
Z^{\prime }
\end{array}
\right)
\end{array}
\end{equation*}
such that the matrices of $\rho _{x}^{\prime }(a)$ and $\rho
_{x}^{\prime }(b)$ are respectively:
\begin{equation*}
m_{x}(a)=\left(
\begin{array}{ccc}
2x^{2}-1 & -2x\sqrt{1-x^{2}} & 0 \\
2x\sqrt{1-x^{2}} & 2x^{2}-1 & 0 \\
0 & 0 & 1
\end{array}
\right)
\end{equation*}
and
\begin{equation*}
m_{x}(b)=
\begin{pmatrix}
2x^{2}-1 & \frac{x-2x^{3}}{\sqrt{1-x^{2}}} & x\sqrt{\frac{3-4x^{2}}{x^{2}-1}}
\\
\frac{x(2x^{2}-1)}{\sqrt{1-x^{2}}} & \frac{1+2x^{2}-4x^{4}}{2-2x^{2}} &
\frac{\sqrt{4x^{2}-3}(2x^{2}-1)}{2(1-x^{2})} \\
x\sqrt{\frac{3-4x^{2}}{x^{2}-1}} & \frac{\sqrt{4x^{2}-3}(2x^{2}-1)}{
2(1-x^{2})} & \frac{1-2x^{2}}{2x^{2}-2}
\end{pmatrix}.
\end{equation*}
The maps $\rho _{x}^{\prime }(a)$ and $\rho _{x}^{\prime }(b)$ are right
(spherical) rotations on $H_{0}\cong E^{1,2}$ of angle $\alpha $ around the
time-like axes $A^{-}$ and $B^{-}$ where $x=A^{+}=B^{+}=\cos \frac{\alpha }{2
}$. See Figure \ref{fambmi}.

\begin{figure}[ht]
\epsfig{file=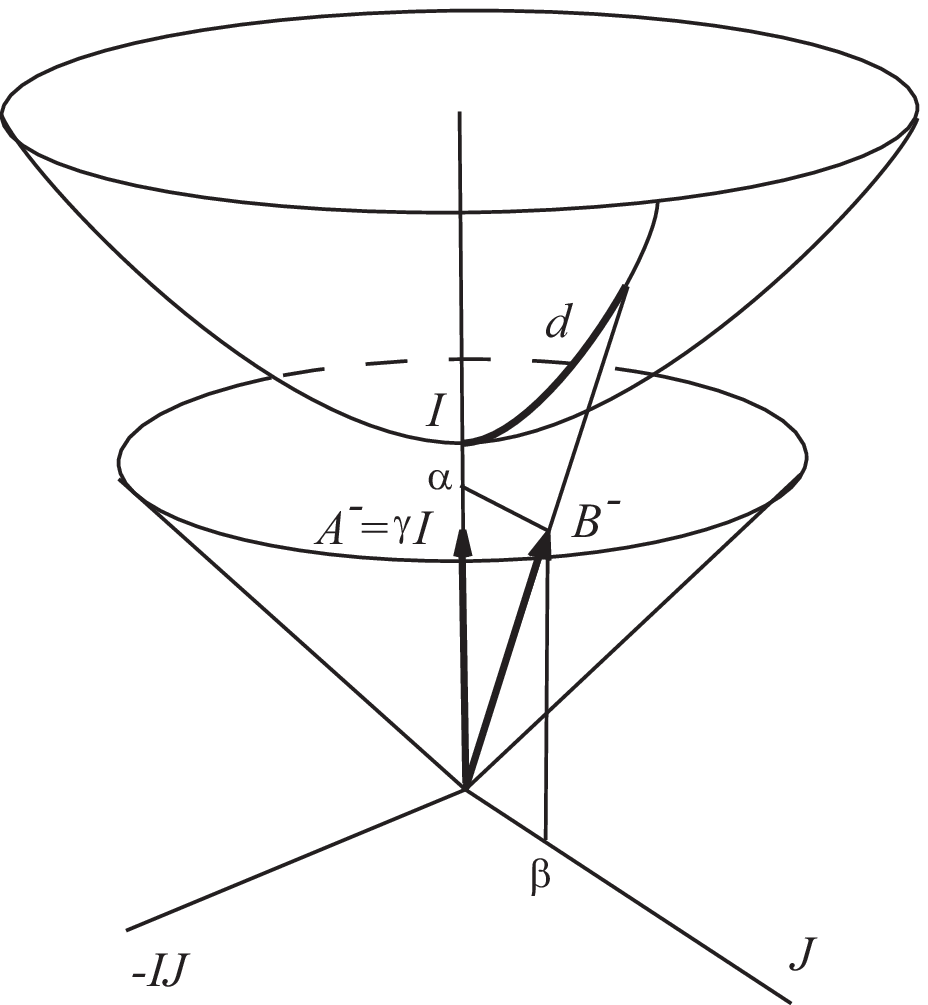,height=4.5cm} \caption{Case 3.}\label{fambmi}
\end{figure}

Moreover, $\rho _{x}^{\prime }(C)$ is the identity. Hence $\rho _{x}^{\prime
}:G(3_{1})\longrightarrow SO^{0}(1,2)$ factors through the group $C_{2}\ast
C_{3}=|F,D;F^{3}=D^{2}=1|$. Thus $\rho _{x}^{\prime }(F)=\rho _{x}^{\prime
}(ab)$ is a $3$-fold rotation and $\rho _{x}^{\prime }(D)=\rho _{x}^{\prime
}(aba)$ is a $2$-fold rotation.

The distance $d$ (measured in the hyperbolic plane) between the axes of $\rho _{x}^{\prime }(a)$ and $\rho _{x}^{\prime }(b)$ is given by
\begin{equation*}
\cosh d=\frac{y}{u}=\frac{x^{2}-\frac{1}{2}}{-x^{2}+1}
\end{equation*}
The geometrical meaning of the representation $\rho _{x}^{\prime
}:G(3_{1})\longrightarrow SO^{0}(1,2)\cong Iso^{+}(\mathbb{H}^{2})$ for $x\in (\frac{\sqrt{3}}{2},1)$ is the following.

\begin{figure}[ht]
\epsfig{file=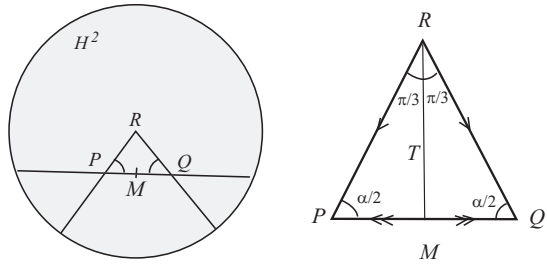,height=3.5cm} \caption{The hyperbolic cone-manifold $S_{\frac{2\pi }{2},
\frac{2\pi }{3},\alpha }^{2}$.}\label{s23alfahy}
\end{figure}

\begin{theorem}
\bigskip \label{tgenso12}The image of $\rho _{x}^{\prime }$, $\frac{\sqrt{3}
}{2}<x<1$, is a subgroup of $SO^{o}(1,2)$ isomorphic to the holonomy of the
2-dimensional hyperbolic cone-manifold $S_{\frac{2\pi }{2},\frac{2\pi }{3}
,\alpha }^{2}$, where $x=\cos \frac{\alpha }{2}$ ($\frac{2\pi }{6}>\alpha >0$
).

\begin{proof}
The 2-dimensional hyperbolic cone-manifold $S_{\frac{2\pi }{2},\frac{2\pi }{3
},\alpha }^{2}$ is the result of identifying the edges of the hyperbolic
triangle $T(\frac{2\pi }{3},\frac{\alpha }{2},\frac{\alpha }{2})$ as in
Figure \ref{s23alfahy} . The holonomy of $S_{\frac{2\pi }{2},\frac{2\pi }{3}
,\alpha }^{2}$ is generated by rotations of angle $\alpha $ in the vertices $
P$ and $Q$. The distance $d^{\prime }$ between $P$ and $Q$ is calculated by
hyperbolic trigonometry:
\begin{equation*}
\cosh d^{\prime }=\frac{\cos ^{2}\frac{\alpha }{2}+\cos \frac{2\pi }{3}}{
\sin ^{2}\frac{\alpha }{2}}=\frac{x^{2}-\frac{1}{2}}{-x^{2}+1}=\cosh d
\end{equation*}
Thus $d^{\prime }=d$. Therefore the points $P$ and $Q$ are the intersection
with the upper sheet of the two sheeted hyperboloid ( model of the
2-dimensional hyperbolic plane) of the axis $A^{-}$ and $B^{-}$of the two
generators of the subgroup $\rho _{x}^{\prime }(G(3_{1})).$
\end{proof}
\end{theorem}

\begin{remark}
The point $R$ (resp. $M$ the mid-point between $P$ and $Q$) is the
intersection with $\mathbb{H}^{2}$ of the axis of the $3$-fold (resp. $2$
-fold) rotation $\rho _{x}^{\prime }(F)=\rho _{x}^{\prime }(ab)$ (resp. $\rho _{x}^{\prime }(D)=\rho _{x}^{\prime }(aba)$).
\end{remark}

\subsection{Case 4}
Segment (2.5).
\[(x,y)=(\pm 1,\frac{1}{2})\Longleftrightarrow \left\{
1-x^{2}=0,(1-x^{2})^{2}<y^{2}\right\}. \]

There exists an irreducible c-representation $\rho
_{x}:G(3_{1})\longrightarrow SL(2,\mathbb{R})=U_{1}\subset \left( \frac{-1,1}{\mathbb{R}}\right) $ realizing $(x,y)$, unique up to conjugation in $SL(2,\mathbb{R})$, such that
\begin{equation*}
\begin{array}{l}
\rho _{x}(a)=A=\pm 1+i+j \\
\rho _{x}(b)=B=\pm 1+\frac{1}{4}(i-j)
\end{array}
\end{equation*}
The composition of $\rho _{x}$ with $c:SL(2,\mathbb{R})\rightarrow
SO^{0}(1,2)\cong Iso^{+}(\mathbb{H}^{2})$, defines the representation $\rho
_{x}^{\prime }=c\circ \rho _{x}:G(3_{1})\longrightarrow SO^{0}(1,2)$ such
that the matrices of $\rho _{x}^{\prime }(a)$ and $\rho _{x}^{\prime }(b)$
are respectively:
\begin{equation*}
m_{x}(a)=m(-1,1;\pm 1,1,1,0)=\left(
\begin{array}{ccc}
1 & -2 & 2 \\
2 & -1 & 2 \\
2 & -2 & 3
\end{array}
\right)
\end{equation*}
and%
\begin{equation*}
m_{x}(b)=m(-1,1;\pm 1,\frac{1}{4},-\frac{1}{4},0)=
\begin{pmatrix}
1 & -\frac{1}{2} & -\frac{1}{2} \\
\frac{1}{2} & \frac{7}{8} & -\frac{1}{8} \\
-\frac{1}{2} & \frac{1}{8} & \frac{9}{8}
\end{pmatrix}
\end{equation*}
where $\left\{ X,Y,Z\right\} $ is the coordinate system associated to the
basis $\left\{ -ij,j,i\right\} $. The maps $\rho _{x}^{\prime }(a)$ and $\rho _{x}^{\prime }(b)$ are parabolic rotations on $H_{0}\cong E^{1,2}$
around the light-like axes $A^{-}$ and $B^{-}$. See Figure \ref{fambmiii}.

\begin{figure}[ht]
\epsfig{file=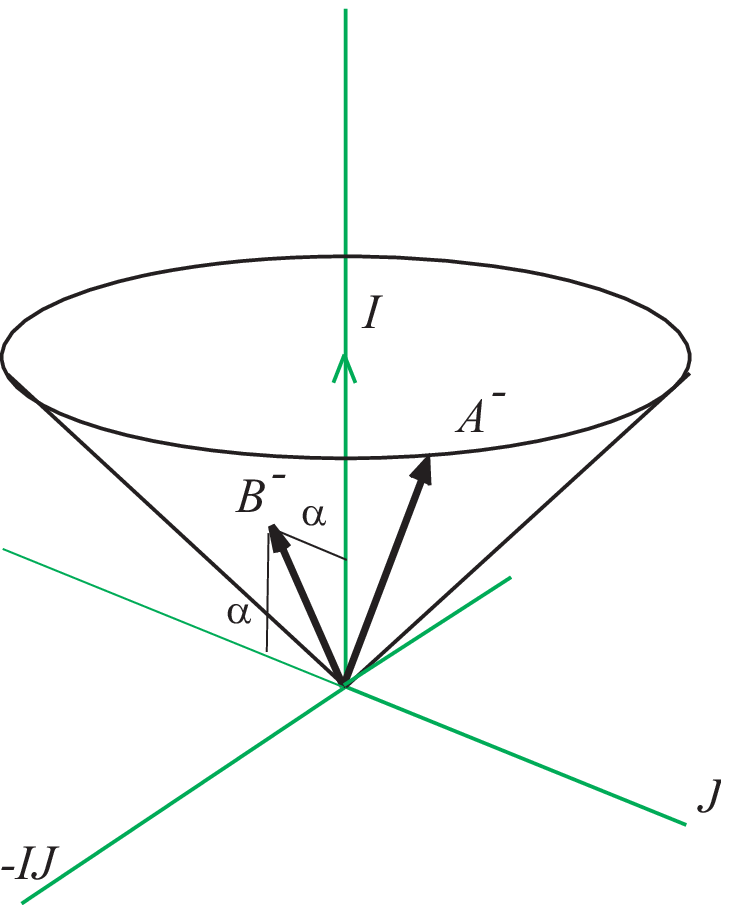,height=4.5cm} \caption{Case 4. }\label{fambmiii}
\end{figure}

Therefore the image of $\rho
_{1}^{\prime }:G(3_{1})\longrightarrow Iso^{+}(\mathbb{H}^{2})$ is conjugate
to the action of the modular group $PSL(2,\mathbb{Z})$ in $\mathbb{H}^{2}$.
Thus

\begin{figure}[ht]
\epsfig{file=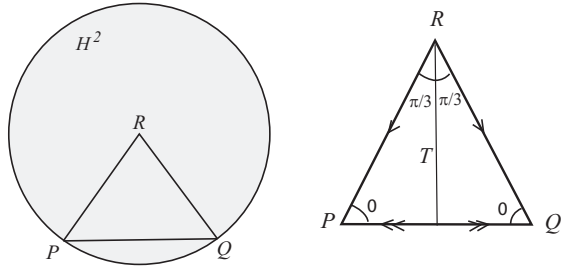,height=3.5cm} \caption{The hyperbolic orbifold $
S_{2,3,\infty }^{2}$.}\label{s23alfay4}
\end{figure}

\begin{theorem}
The image of $\rho _{1}^{\prime }:G(3_{1})\longrightarrow SO^{0}(1,2)$ acts
on the hyperbolic plane $\mathbb{H}^{2}$ as the holonomy of the hyperbolic
orbifold $S_{2,3,\infty }^{2}$.
\end{theorem}

\subsection{Case 5}
Region (2.2).
\[x\in (-\infty ,-1)\cup (1,\infty )\Longleftrightarrow \left\{
1-x^{2}<0,(1-x^{2})^{2}<y^{2}\right\}\quad ( y>0).\]

There exists an irreducible c-representation $\rho
_{x}:G(3_{1})\longrightarrow SL(2,\mathbb{R})=U_{1}\subset \left( \frac{-1,1}{\mathbb{R}}\right) $ realizing $(x,y)$, unique up to conjugation in $SL(2,\mathbb{R})$, such that
\begin{equation*}
\begin{array}{l}
\rho _{x}(a)=A=x+\sqrt{x^{2}-1}j,\quad \sqrt{x^{2}-1}>0 \\
\rho _{x}(b)=B=x-\frac{1}{2}\sqrt{\frac{4x^{2}-3}{x^{2}-1}}i-\frac{2x^{2}-1}{
2\sqrt{x^{2}-1}}j
\end{array}
\end{equation*}
The composition of $\rho _{x}$ with $c:SL(2,\mathbb{R})\rightarrow
SO^{0}(1,2)\cong Iso^{+}(\mathbb{H}^{2})$, defines the representation $\rho
_{x}^{\prime }=c\circ \rho _{x}:G(3/1)\longrightarrow SO^{0}(1,2)$ such that
the matrices of $\rho _{x}^{\prime }(a)$ and $\rho _{x}^{\prime }(b)$ are
respectively:

\begin{equation*}
m_{x}(a)=\left(
\begin{array}{cccc}
2x^{2}-1 & 0 & 2x\sqrt{x^{2}-1} & 0 \\
0 & 1 & 0 & \frac{\left( 3-4x^{2}\right) \sqrt{x^{2}-1}}{4x} \\
2x\sqrt{x^{2}-1} & 0 & 2x^{2}-1 & 0
\end{array}
\right)
\end{equation*}

\begin{equation*}
m_{x}(b)=
\begin{pmatrix}
2x^{2}-1 & x\sqrt{\frac{-3+4x^{2}}{x^{2}-1}} & -\frac{x(2x^{2}-1)}{\sqrt{
x^{2}-1}} \\
-x\sqrt{\frac{-3+4x^{2}}{x^{2}-1}} & \frac{1-2x^{2}}{2x^{2}-2} & \frac{\sqrt{
-3+4x^{2}}(2x^{2}-1)}{2(x^{2}-1)} \\
-\frac{x(2x^{2}-1)}{\sqrt{x^{2}-1}} & \frac{\sqrt{-3+4x^{2}}(2x^{2}-1)}{
2(x^{2}-1)} & \frac{1+2x^{2}-4x^{4}}{2-2x^{2}}
\end{pmatrix}
\end{equation*}

The maps $\rho _{x}^{\prime }(a)$ and $\rho _{x}^{\prime }(b)$ are
hyperbolic rotations on $H_{0}\cong E^{1,2}$ moving $\delta $ along the
polars of the space-like vectors $A^{-}$ and $B^{-}$ where $
x=A^{+}=B^{+}=\cosh \frac{\delta }{2}$. See Figure \ref{fcaseii}.

\begin{figure}[ht]
\epsfig{file=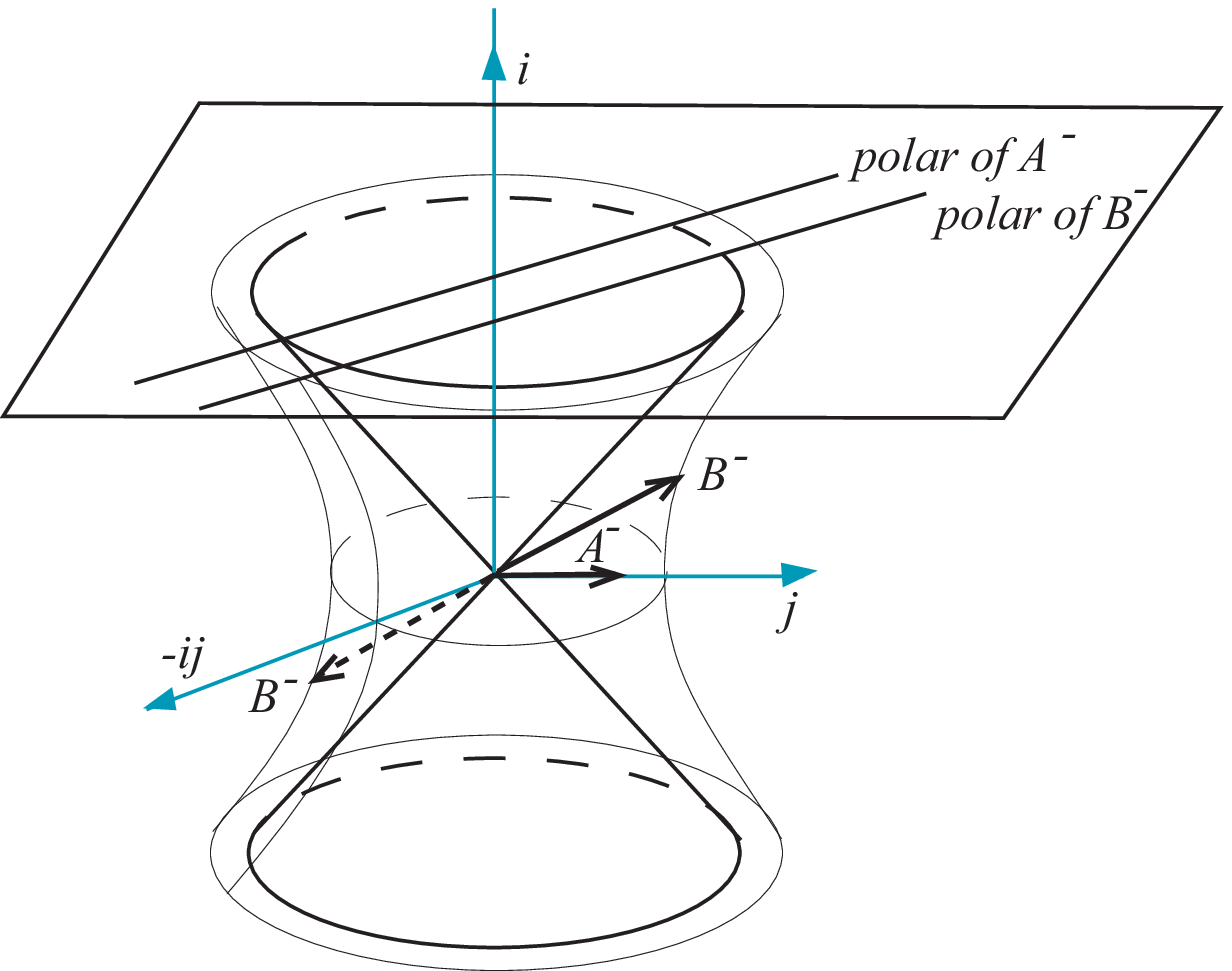,height=4.5cm} \caption{Case 5.}\label{fcaseii}
\end{figure}

Moreover, $\rho _{x}^{\prime }(C)$ is the identity. Hence $\rho
_{x}^{\prime }:G(3_{1})\longrightarrow SO^{0}(1,2)$ factors through the
group $C_{2}\ast C_{3}=|F,D;F^{3}=D^{2}=1|$. Thus $\rho _{x}^{\prime
}(F)=\rho _{x}^{\prime }(ab)$ is a $3$-fold rotation and $\rho _{x}^{\prime
}(D)=\rho _{x}^{\prime }(aba)$ is a $2$-fold rotation.

The distance $d$ (measured in the hyperbolic plane) between the polars of
the axes of $\rho _{x}^{\prime }(a)$ and $\rho _{x}^{\prime }(b)$ is given
by
\begin{equation*}
\cosh d=\frac{y}{x^{2}-1}=\frac{x^{2}-\frac{1}{2}}{x^{2}-1}
\end{equation*}
The geometrical meaning of the representation $\rho _{x}^{\prime
}:G(3_{1})\longrightarrow SO^{0}(1,2)\cong Iso^{+}(\mathbb{H}^{2})$ for $x\in (1,\infty )$ is the following.

\begin{figure}[ht]
\epsfig{file=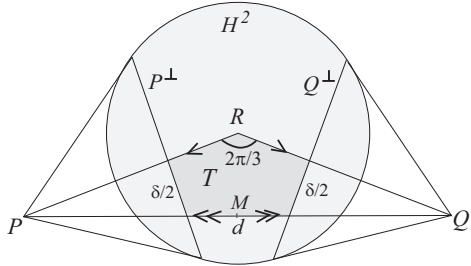,height=3.5cm} \caption{The hyperbolic orbifold $O_{2,3,\delta }$.}\label{s23alfay5}
\end{figure}

\begin{theorem}
\bigskip \label{tgenso122}The image of $\rho _{x}^{\prime }$, $1<x<\infty $,
is a subgroup of $SO^{0}(1,2)$ isomorphic to the holonomy of the
2-dimensional hyperbolic orbifold $O_{2,3,\delta }^{2}$ where $O^{2}$ is an
open disk and $\delta $ is the length of the closed geodesic at the end of $O^{2}$, where $x=\cosh \frac{\delta }{2}$.

\begin{proof}
The 2-dimensional hyperbolic cone-manifold $O_{2,3,\delta }^{2}$ is the
result of identifying the edges of the hyperbolic triangle $T(\frac{2\pi }{3},\frac{\delta }{2},\frac{\delta }{2})$ as in Figure \ref{s23alfay5}, where $P$ and $Q$ are ultrainfinite points such that the length of the segments $T(\frac{2\pi }{3},\frac{\delta }{2},\frac{\delta }{2})\cap P^{\perp }$ and $T(\frac{2\pi }{3},\frac{\delta }{2},\frac{\delta }{2})\cap Q^{\perp }$ are
both $\frac{\delta }{2}$ ($P^{\perp }$ denotes the polar of $P$). The
holonomy of $O_{2,3,\delta }^{2}$ is generated by hyperbolic rotations of
length $\delta $ around the vertices $P$ and $Q$. The distance $r$ between $P^{\perp }$ and $Q^{\perp }$ is calculated by hyperbolic trigonometry:
\begin{equation*}
\cosh r=\frac{\cosh ^{2}\frac{\delta }{2}+\cos \frac{2\pi }{3}}{\sinh ^{2}
\frac{\delta }{2}}=\frac{x^{2}-\frac{1}{2}}{x^{2}-1}=\cosh d
\end{equation*}
Hence $r=d$.
\end{proof}
\end{theorem}

\begin{remark}
The point $R$ (resp. the mid-point between $P^{\perp }$ and $Q^{\perp }$) is
the intersection with $\mathbb{H}^{2}$ of the axis of the $3$-fold (resp. $2$-fold) rotation.
\end{remark}

\subsection{Some particular cases}
\begin{description}
\item[Case1] We have seen that the image of $\rho _{x}^{\prime
}:G(3_{1})\longrightarrow SO(3)$ , $0\leqslant x<\frac{\sqrt{3}}{2}$, is
isomorphic to the holonomy of the 2-dimensional spherical cone-manifold $S_{\frac{2\pi }{2},\frac{2\pi }{3},\alpha }^{2}$. The orbifolds among these
cone-manifolds are $S_{232}$, $S_{233}$, $S_{234}$, $S_{235}$. The
fundamental groups of these orbifolds are the group of direct symmetries of
the equilateral triangle (in $E^{3}$); the tetrahedron; the octahedron; and
the icosahedron, respectively. These groups are isomorphic to, respectively,
$\Sigma _{3}$, $A_{4}$, $\Sigma _{4}$ and $A_{5}$.

\item[Case 2] The image of the almost-irreducible c-representation $\rho_{\frac{\sqrt{3}}{2}}^{\prime }:G(3_{1})\longrightarrow PSL(2,\mathbb{C})$
acts on the Euclidean plane $\mathbb{C}$ as the holonomy of the Euclidean
crystallographic orbifold $S_{236}$.

\item[Case 3] The image of $\rho _{x}^{\prime }$, $\frac{\sqrt{3}}{2}<x<1$,
is a subgroup of $SO^{o}(1,2)$ isomorphic to the holonomy of the
2-dimensional hyperbolic cone-manifold $S_{\frac{2\pi }{2},\frac{2\pi }{3}
,\alpha }^{2}$, where $x=\cos \frac{\alpha }{2}$. The orbifolds among these
cone-manifolds are $S_{23p}$, for all $p\geq 7$. The orbifold $S_{237}$ is
the (orientable) hyperbolic orbifold of minimal area.

\item[Case 4] The image of $\rho _{1}^{\prime }:G(3_{1})\longrightarrow
SO^{0}(1,2)$ acts on the hyperbolic plane $\mathbb{H}^{2}$ as the holonomy
of the hyperbolic open orbifold (with finite volume) $S_{23\infty }$.

\item[Case 5] The image of $\rho _{x}^{\prime }$, $1<x<\infty $, is a
subgroup of $SO^{0}(1,2)$ isomorphic to the holonomy of the 2-dimensional
open hyperbolic orbifold (with infinite volume) $O_{23\delta }$ where $O$ is
an open disk and $\delta $ is the length of the closed geodesic at the end
of $O$, where $x=\cosh \frac{\delta }{2}$.
\end{description}

\begin{theorem}
\label{tsubgrupolibre} Every image of $\rho _{x}^{\prime
}:G(3_{1})\longrightarrow SO^{0}(1,2)$, where $x\geq 1$ contains a two
generator free subgroup of finite index.
\end{theorem}

\begin{proof}
Theorem \ref{tgenso122} shows that the image of $\rho _{x}^{\prime
}:G(3_{1})\longrightarrow SO^{0}(1,2)$, $1<x<\infty $, is a subgroup of $
SO^{0}(1,2)$ isomorphic to the holonomy of the 2-dimensional open hyperbolic
orbifold (with infinite volume) $O_{23\delta }$ where $O$ is an open disk
and $\delta $ is the length of the closed geodesic at the end of $O$, where $x=\cosh \frac{\delta }{2}$. Figure \ref{s23alfay5}. The image of $\rho
_{1}^{\prime }:G(3_{1})\longrightarrow SO^{0}(1,2)$, is a subgroup of $SO^{0}(1,2)$ isomorphic to the holonomy of the 2-dimensional open hyperbolic
orbifold (with finite volume) $S_{\frac{2\pi }{2},\frac{2\pi }{3},\infty
}^{2}$ with underlying space a punctured 2-sphere. Figure \ref{s23alfay4}.
Let us denote $S_{\frac{2\pi }{2},\frac{2\pi }{3},\infty }^{2}$ by $O_{230}$
to unify notation.

Consider the homomorphism
\begin{equation*}
\begin{array}{cccc}
\Omega : & \pi _{1}^{o}(O_{23\delta })=|F,D;F^{3}=D^{2}=1| & \longrightarrow
& \Sigma _{6} \\
& F & \rightarrow & (123)(456) \\
& D & \rightarrow & (15)(24)(36)
\end{array}
\end{equation*}
where $\delta \geq 0$. It defines a covering of orbifolds
\begin{equation*}
p_{\Omega }:O_{2\delta ,2\delta ,2\delta }\overset{6:1}{\longrightarrow }
O_{23\delta }
\end{equation*}
where $O_{2\delta ,2\delta ,2\delta }$ is a 2-sphere with 3 holes and the
length of the closed geodesic at every hole is $2\delta $ if $\delta >0$ and
$O_{0,0,0}$ is a 3-punctured 2-sphere. Figure \ref{fo2d2d2d}.

\begin{figure}[ht]
\epsfig{file=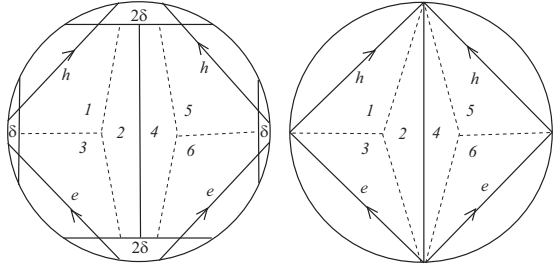,height=3.5cm} \caption{The hyperbolic manifold $O_{2
\delta ,2\delta ,2\delta }$, $\delta \geq 0$.}\label{fo2d2d2d}
\end{figure}

The fundamental group of $O_{2\delta
,2\delta ,2\delta }$, $\pi _{1}^{o}(O_{2\delta ,2\delta ,2\delta })$, is a
free subgroup of $\rho _{x}^{\prime }(G(3_{1}))$ generated by $g_{1}=F^{-2}DFD^{-1}=FDFD$ and $g_{2}=F^{-1}DF^{2}D^{-1}=F^{-1}DF^{-1}D$. We
can write the two generators of $\pi _{1}^{o}(O_{2\delta ,2\delta ,2\delta
}) $ in terms of $\rho _{x}^{\prime }(a)=A$ and $\rho _{x}^{\prime }(b)=B$ ,
because $F=AB$ and $D=ABA=BAB\mathbf{.}$ Then
\begin{eqnarray*}
g_{1} &=&FDFD=\underset{B^{-1}}{\underbrace{ABABA}}\underset{B^{-1}}{
\underbrace{ABABA}}\mathbf{=}B^{-1}B^{-1} \\
g_{2} &=&F^{-1}DF^{-1}D=B^{-1}A^{-1}ABAB^{-1}A^{-1}ABA\mathbf{=}AA
\end{eqnarray*}
We conclude that the subgroup of $\rho _{x}^{\prime }(G(3_{1}))$ generated
by $\{\rho _{x}^{\prime }(b^{-2}),$ $\rho _{x}^{\prime }(a^{2})\}$ is an
index 6 free subgroup of rank 2 generated by $\{B^{-2},A^{2}\}$.
\end{proof}

\section{The representations of $G(3_{1})$ in $A(H)$}

In this section we obtain all the representations of the trefoil knot group
in the affine isometry group $A(H)$ of a quaternion algebra $H$. They
include all the representations in the 3-dimensional affine Euclidean
isometries $\mathcal{E}(E^{3})$ and all the representations in the
3-dimensional affine Lorentz isometries $\mathcal{L}(E^{1,2})$. The
representations in the affine isometry group of a quaternion algebra are
affine deformations of representations in the unit quaternions group of $H$,
computed in the above section.

Let
\begin{equation*}
\begin{array}{llll}
\rho : & G & \longrightarrow & A(H) \\
& a & \rightarrow & \rho (a)=(sA^{-},A) \\
& b & \rightarrow & \rho (b)=(sB^{-}+(A^{-}B^{-})^{-},B)
\end{array}
\end{equation*}
be a representation of $G$ in the affine group of a quaternion algebra $H$.
The composition of $\rho $ with the projection $\pi _{2}$ on the second
factor of $A(H)=H_{0}\rtimes U_{1}$ gives the linear part of $\rho $ and it
is a representation $\widehat{\rho }$ on the group of unitary quaternions.
\begin{equation*}
\begin{array}{llll}
\widehat{\rho }=\pi _{2}\circ \rho : & G & \longrightarrow & U_{1} \\
& a & \rightarrow & A \\
& b & \rightarrow & B
\end{array}
\end{equation*}
The composition of $\rho $ with the projection $\pi _{1}$ on the first
factor of $A(H)=H_{0}\rtimes U_{1}$ gives the translational part of $\rho :$
\begin{equation*}
\begin{array}{llll}
v_{\rho }=\pi _{1}\circ \rho : & G & \longrightarrow & H_{0} \\
& a & \rightarrow & sA^{-} \\
& b & \rightarrow & sB^{-}+(A^{-}B^{-})^{-}
\end{array}
\end{equation*}
Therefore $\widehat{\rho }$ corresponds to a point in the character variety $V(\mathcal{I}_{G})$ of representations of $G$ and it is determined by the
characters $x$ and $y$. We are interested in the classes of affine
deformations up to conjugation in the equiform group $\mathcal{E}
q(H)=H_{0}\rtimes U$, where $U$ is the group of invertible elements . Each
of these classes is determined by the parameter $s.$

The relation \eqref{erel2} in the case of the Trefoil knot group shows that
the parameter $s$ satisfies the equation
\begin{equation}
4x^{2}+4sx-3=0  \label{epoly2}
\end{equation}

\begin{theorem}
\bigskip Every representation $\rho _{x}:G(3/1)\longrightarrow
A(H)=H_{0}\rtimes U_{1}$ defined by%
\begin{equation*}
\begin{array}{l}
\rho _{x}(a)=(sA^{-},A) \\
\rho _{x}(b)=((sB^{-}+(A^{-}B^{-})^{-},B)
\end{array}
\end{equation*}
where $A$ and $B$ are independent unit quaternions, factors through the
group $C_{2}\ast C_{3}.$
\end{theorem}

\begin{proof}
The group $G(3_{1})$ has also the presentation $\left\vert
F,D;F^{3}=D^{2}\right\vert $ (Recall \eqref{genfd}.)

The group $C_{2}\ast C_{3}.$ is a quotient of $G(3_{1})=\left\vert
F,D;F^{3}=D^{2}\right\vert $
\begin{equation*}
q:G(3/1)=\left\vert F,D;F^{3}=D^{2}\right\vert \longrightarrow C_{2}\ast
C_{3}=\left\vert F,D;F^{3}=D^{2}=1\right\vert
\end{equation*}
Then to prove that there exists a homomorphism
\begin{equation*}
\lambda _{x}:C_{2}\ast C_{3}=\left\vert F,D;F^{3}=D^{2}=1\right\vert
\longrightarrow A(H)
\end{equation*}
such that $\lambda _{x}\circ q=\rho _{x}$ it is enough to show that $\rho
_{x}(D^{2})=(0,1)$. Recall that the element $C=D^{2}=F^{3}$ belongs to the
center of the group $G(3_{1})$, therefore $\rho _{x}(C)$ commutes with every
element of $\rho _{x}(G(3_{1})),$ in particular with $\rho _{x}(a)$ and $
\rho _{x}(b)$. Let us consider first the linear part $\widehat{\rho }_{x}=\pi
_{2}\circ \rho _{x}:G(3_{1})\longrightarrow U_{1}$. Because $A$ and $B$ are
independent, $A^{-}\neq \pm B^{-}$, the only element of $U_{1}$ commuting
with $A$ and $B$ is the identity 1. Therefore $\widehat{\rho }_{x}(C)=1$,
and $\rho _{x}(C)$ is a translation. If $\rho _{x}(C)=(v,1)$ and $\rho
_{x}(a)=(sA^{-},A)$ commute then
\begin{eqnarray*}
(v,1)(sA^{-},A) &=&(sA^{-},A)(v,1)\Longrightarrow
(v+sA^{-},A)=(sA^{-}+A\circ v,A) \\
&\Longrightarrow &v=A\circ v=AvA^{-1}
\end{eqnarray*}
we deduce that $v=0$ or $v$ and $A^{-}$ are dependent. An analogous
computation with $\rho _{x}(b)$ yields $v=0$ or $v$ and $B^{-}$ are
dependent. As $A$ and $B$ are independent, $v=0.$ We have proved that $\rho
_{x}(C)=(0,1)$.
\end{proof}

Moreover, every homomorphism $\lambda :C_{2}\ast C_{3}=\left\vert
F,D;F^{3}=D^{2}=1\right\vert \longrightarrow A(H)$ induces a representation $
\rho :G(3/1)\longrightarrow A(H)$ such that $\rho =\lambda \circ q$.
Therefore we can apply the results on representations of $G(3_{1})$ in $A(H)$
to representations of $C_{2}\ast C_{3}$ in $A(H).$

\begin{corollary}
\label{cor1}Every non cyclic subgroup $S$ in $A(H)$ generated by two
isometries $\mu \neq 1,$ $\upsilon \neq 1$ such that $\mu ^{2}=\upsilon
^{3}=1$ is necessarily generated by two conjugate elements in $A(H)$.
\end{corollary}

\begin{proof}
The subgroup $S$ is the image of a representation
\begin{equation*}
\begin{array}{cccc}
\lambda : & C_{2}\ast C_{3}=\left\vert F,D;F^{3}=D^{2}=1\right\vert &
\longrightarrow & A(H) \\
& D & \rightarrow & \mu \\
& F & \rightarrow & \upsilon
\end{array}
\end{equation*}
Then it defines a representation $\rho =\lambda \circ
q:G(3/1)\longrightarrow A(H)$ with the same image $S.$ As $G(3_{1})$ is
generated by two conjugate elements $a$ and $b$, the group $S$ is generated by
the two conjugate elements $\rho (a)$ and $\rho (b)$.
\end{proof}

\subsection{The Hamilton quaternion algebra $\mathbb{H}=\left( \frac{-1,-1}{\mathbb{R}}\right) $}
Let us study first the affine deformations of representations corresponding to
points in the character variety belonging to Case 1: $x\in (\frac{-\sqrt{3}}{
2},\frac{\sqrt{3}}{2})\Longleftrightarrow \left\{
1-x^{2}>0,(1-x^{2})^{2}>y^{2}\right\} ,$ where the quaternion algebra is the
Hamilton quaternion algebra $\mathbb{H}=\left( \frac{-1,-1}{\mathbb{R}}
\right) $. Therefore $U_{1}=S^{3}$, $H_{0}=E^{3}$.

Recall that in this case, the geometrical meaning of parameters $
x,y,u,s,\sigma $ is the following.

\begin{figure}[ht]
\epsfig{file=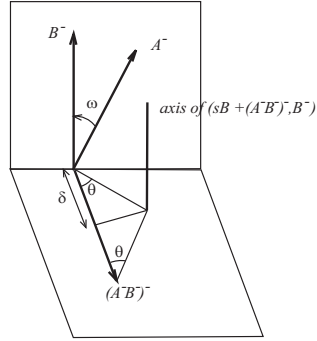,height=4.5cm} \caption{The geometrical meaning of parameters in $E^{3}$.}\label{fdelta}
\end{figure}

Let $\alpha $ be the angle of the right rotation around the axis $A^{-}$
corresponding to the action of $A=A^{+}+A^{-}$ in $H_{0}\cong E^{3}$. Let $
\omega $ be the angle defined by $A^{-}$ and $B^{-}$. See Figure \ref{fdelta}
. Then
\begin{eqnarray}
x &=&A^{+}=B^{+}=\cos \frac{\alpha }{2}  \notag \\
u &=&N(A^{-})=N(B^{-})=1-x^{2}=\sin ^{2}\frac{\alpha }{2}  \notag \\
y &=&-(A^{-}B^{-})^{+}=\left\langle A^{-},B^{-}\right\rangle =u\cos \omega
\label{eomega}
\end{eqnarray}
The shift $\sigma $ of the element $(sA^{-},A)$ or any element of $A(H)$
conjugate to $(sA^{-},A)$ is given by
\begin{equation}
\sigma =s\sqrt{u}=s\sqrt{1-x^{2}}  \label{esigma}
\end{equation}
The distance $\delta $ between the axes of $(sA^{-},A)$ and $
(sB^{-}+(A^{-}B^{-})^{-},B)$ is half of the length of the vector $
(A^{-}B^{-})^{-}$. See Figure \ref{fdelta}.
\begin{equation}
N((A^{-}B^{-})^{-})=-y^{2}+u^{2}\quad \Longrightarrow \quad \delta =\frac{
\sqrt{u^{2}-y^{2}}}{2}  \label{edelta}
\end{equation}

The following Theorem gives the affine deformations of the subgroups of $
SO(3,\mathbb{R})$ which are the image of representations of $G(3_{1}),$ or
equivalently the affine deformations of the holonomies of the spherical
conemanifold $S_{\frac{2\pi }{2},\frac{2\pi }{3},\alpha }^{2}$, $\frac{2\pi
}{6}<\alpha <\frac{10\pi }{6}$.

\begin{theorem}
\label{tgeneucl} For each $x\in (-\sqrt{3}/2,\sqrt{3}/2),$ there exists a
representation $\rho _{x}:G(3/1)\longrightarrow \mathcal{E}(E^{3})$ unique
up to conjugation in $\mathcal{E}(E^{3})$ such that
\begin{equation*}
\begin{array}{l}
\rho _{x}(a)=(sA^{-},A) \\
\rho _{x}(b)=((sB^{-}+(A^{-}B^{-})^{-},B)
\end{array}
\end{equation*}
where
\begin{equation*}
\begin{array}{l}
A=x+\frac{2x^{2}-1}{2\sqrt{1-x^{2}}}i+\frac{1}{2}\sqrt{\frac{3-4x^{2}}{
1-x^{2}}}j \\
B=x+\sqrt{1-x^{2}}i,\quad \sqrt{1-x^{2}}>0
\end{array}
\end{equation*}
In affine linear notation, where $\left\{ X,Y,Z\right\} $ is the coordinate
system associated to the basis $\left\{ -ij,j,i\right\} $
\begin{equation*}
\begin{array}{l}
\rho _{x}(a)=M_{x}(a)\left(
\begin{array}{c}
X \\
Y \\
Z \\
1
\end{array}
\right) =\left(
\begin{array}{c}
X^{\prime } \\
Y^{\prime } \\
Z^{\prime } \\
1
\end{array}
\right) \\
\rho _{x}(b)=M_{x}(b)\left(
\begin{array}{c}
X \\
Y \\
Z \\
1
\end{array}
\right) =\left(
\begin{array}{c}
X^{\prime } \\
Y^{\prime } \\
Z^{\prime } \\
1
\end{array}
\right)
\end{array}
\end{equation*}
where
\begin{equation}
M_{x}(a)=\left(
\begin{array}{cccc}
2x^{2}-1 & \frac{x-2x^{3}}{\sqrt{1-x^{2}}} & x\sqrt{\frac{4x^{2}-3}{x^{2}-1}}
& 0 \\
\frac{x(2x^{2}-1)}{\sqrt{1-x^{2}}} & \frac{-4x^{4}+2x^{2}+1}{2-2x^{2}} &
\frac{(1-2x^{2})\sqrt{3-4x^{2}}}{2(x^{2}-1)} & \frac{\sqrt{\left(
3-4x^{2}\right) ^{3}}}{8x\sqrt{1-x^{2}}} \\
-x\sqrt{\frac{4x^{2}-3}{x^{2}-1}} & \frac{(1-2x^{2})\sqrt{3-4x^{2}}}{
2(x^{2}-1)} & \frac{1-2x^{2}}{2x^{2}-2} & \frac{-8x^{4}+10x^{2}-3}{8x\sqrt{
1-x^{2}}} \\
0 & 0 & 0 & 1
\end{array}
\right)  \label{emxa}
\end{equation}
\begin{equation}
M_{x}(b)=
\begin{pmatrix}
2x^{2}-1 & -2x\sqrt{1-x^{2}} & 0 & -\frac{1}{2}\sqrt{3-4x^{2}} \\
2x\sqrt{1-x^{2}} & 2x^{2}-1 & 0 & 0 \\
0 & 0 & 1 & \frac{(3-4x^{2})\sqrt{1-x^{2}}}{4x} \\
0 & 0 & 0 & 1
\end{pmatrix}
\label{emxb}
\end{equation}
The distance $\delta $ and the angle $\omega $ between the axis of $\rho
_{x}(a)$ and $\rho _{x}(b)$ are given by
\begin{eqnarray*}
\delta &=&\frac{\sqrt{3-4x^{2}}}{4} \\
\cos \omega &=&\frac{2x^{2}-1}{2-2x^{2}}
\end{eqnarray*}
The shift $\sigma $ is
\begin{equation*}
\sigma =(\frac{3}{4x}-x)\sqrt{1-x^{2}}
\end{equation*}

\begin{proof}
The values of $A$ and $B$ are given by \eqref{ecase1}. From the polynomial \eqref{epoly2}
 we obtain the value of $s$
\begin{equation*}
s=\frac{3-4x^{2}}{4x}
\end{equation*}
The associated representation $\rho _{x}:G(3/1)\longrightarrow \mathcal{E}
(E^{3})$  such that
\begin{equation*}
\begin{array}{l}
\rho _{x}(a)=(sA^{-},A) \\
\rho _{x}(b)=((sB^{-}+(A^{-}B^{-})^{-},B)
\end{array}
\end{equation*}
is unique up to conjugation in $\mathcal{E}(E^{3})$ (\cite[Th.5]{HLM2009}) and $(A^{-}B^{-})^{-}=-\sqrt{(1-x^{2})-y^{2}}ij$. The matrices $M_{x}(a)$ and $M_{x}(b)$
as affine transformations are given by
\begin{equation*}
M_{x}(a)=
\begin{pmatrix}
m_{x}(-1,-1;x,\frac{y}{\sqrt{1-x^{2}}},\frac{\sqrt{(1-x^{2})^{2}-y^{2}}}{
\sqrt{1-x^{2}}},0) &
\begin{array}{c}
0 \\
s\frac{\sqrt{(1-x^{2})^{2}-(y)^{2}}}{\sqrt{1-x^{2}}} \\
s\frac{y}{\sqrt{1-x^{2}}}
\end{array}
\\
\begin{array}{ccc}
\quad 0\quad  & \quad 0\quad  & \quad 0\quad
\end{array}
& 1
\end{pmatrix}
\end{equation*}
\begin{equation*}
M_{x}(b)=
\begin{pmatrix}
m_{x}(-1,-1;x,\sqrt{1-x^{2}},0,0) &
\begin{array}{c}
-\sqrt{(1-x^{2})^{2}-y^{2}} \\
0 \\
s\sqrt{1-x^{2}}
\end{array}
\\
\begin{array}{ccc}
\quad 0\quad  & \quad 0\quad  & \quad 0\quad
\end{array}
& 1
\end{pmatrix}
\end{equation*}
The values of $\delta $, $\cos \omega $ and $\sigma $ are given by \eqref{edelta}, \eqref{eomega} and \eqref{esigma}.
\end{proof}
\end{theorem}

\begin{remark}
\label{rro0}The representation
\begin{equation*}
\begin{array}{cccc}
\widehat{\rho }_{0}: & G(3/1) & \longrightarrow & S^{3} \\
& a & \rightarrow & -\frac{1}{2}i+\frac{\sqrt{3}}{2}j \\
& b & \rightarrow & i
\end{array}
\end{equation*}
does not have any affine deformation, because the polynomial \eqref{epoly2} is
never zero for $x=0$.
\end{remark}

The following theorem states the classification up to conjugation of the non
cyclic subgroups of $\mathcal{E}(E^{3})$ generated by two isometries $\mu
\neq 1,$ $\upsilon \neq 1$ such that $\mu ^{2}=\upsilon ^{3}=1$.

\begin{theorem}
\label{tposiblesH} Every non cyclic subgroup $S$ in the group of direct
Euclidean isometries $\mathcal{E}(E^{3})$ generated by two isometries $\mu
\neq 1,$ $\upsilon \neq 1$ such that $\mu ^{2}=\upsilon ^{3}=1$ is one of
the following

\begin{enumerate}
\item [1.] If $S$ has a fixed point, then $S$ is conjugate in $\mathcal{E}(E^{3})$
to the holonomy of the 2-dimensional spherical conemanifold $S_{\frac{2\pi }{
2},\frac{2\pi }{3},\alpha }^{2}$ , $\frac{2\pi }{6}<\alpha <\frac{10\pi }{6}$
.

\item [2.] If $S$ has no fixed points, then three cases are possible

\begin{enumerate}
\item [2.a)] $S$ is conjugate in $\mathcal{E}(E^{3})$ to the subgroup generated by $
M_{x}(a)$ and $M_{x}(b)$, $x\in (0,\sqrt{3}/2).$ See \eqref{emxa} and \eqref{emxb}.

\item [2.b)]$S$ is conjugate in $\mathcal{E}(E^{3})$ to the natural extension to $
\mathcal{E}(E^{3})$ of the holonomy of the 2-dimensional Euclidean orbifold $
S_{236}^{2}$.

\item [2.c)]$S$ is conjugate in $\mathcal{E}(E^{3})$ to the Euclidean
crystallographic group $P6_{1}$.
\end{enumerate}
\end{enumerate}
\end{theorem}

\begin{proof}
If $S$ has a fixed point, the translational part of $S$ is $0.$ Then $
S\subset SO(3)$, and therefore it is conjugate to the image of
\begin{equation*}
\widehat{\rho }_{x}^{\prime }=c\circ \widehat{\rho }_{x}:G(3/1)
\longrightarrow SO(3)
\end{equation*}
for some $x\in (-\sqrt{3}/2,\sqrt{3}/2)$. Theorem \ref{tgenso3} shows that $
\widehat{\rho }_{x}^{\prime }\left( G(3/1)\right) $ is isomorphic to the
holonomy of the 2-dimensional spherical conemanifold $S_{\frac{2\pi }{2},
\frac{2\pi }{3},\alpha }^{2}$, where $x=\cos \frac{\alpha }{2}$.

If $S$ has no fixed points, we analyze the relative position of the axes
of the generators $\rho (a)$ and $\rho (b)$. If these axes are not parallel,
then $S$ is conjugate to the image of
\begin{equation*}
\rho _{x}:G(3/1)\longrightarrow \mathcal{E}(E^{3})
\end{equation*}
for some $x\in \lbrack 0,\sqrt{3}/2)$. Theorem \ref{tgeneucl} gives the
generators $M_{x}(a)$ and $M_{x}(b)$.

If the axes are parallel, we show in the following that $S$ is conjugate to
the image of an affine deformation $\rho :G(3/1)\longrightarrow \mathcal{E}
(E^{3})$ of a representation
\begin{equation*}
\pi _{2}\circ \rho =\widehat{\rho }^{\prime }=c\circ \widehat{\rho }
:G(3/1)\longrightarrow SO(3)
\end{equation*}
Observe that the linear part of $S$, $\pi _{2}(S)$ is generated by two
conjugate rotations $(\pi _{2}\circ \rho )(a)$ and $(\pi _{2}\circ \rho )(b)$
with the same axis. Therefore $\widehat{\rho }$ is a reducible
representation.
\begin{equation*}
C_{2}\ast C_{3}\overset{\lambda }{\longrightarrow }\mathcal{E}(E^{3})\overset
{\pi _{2}}{\longrightarrow }SO(3)
\end{equation*}
Then $\pi _{2}\circ \lambda $ factors through the abelianized group $C_{6}$
of $C_{2}\ast C_{3}$. Therefore $\pi _{2}(S)$ is a cyclic group of order
dividing $6$. The elements $(\pi _{2}\circ \rho )(a)$ and $(\pi _{2}\circ
\rho )(b)$ are both elements of order $1,2,3$ or $6$. Up to similarity we
assume that the axis of $\rho (a)$ is the $Z$ axis and the axis of $\rho (b)$
is the line parallel to the $Z$ axis through the point $(1,0,0)$. In
affine notation
\begin{equation*}
\rho (a)=
\begin{pmatrix}
\cos \alpha  & -\sin \alpha  & 0 & 0 \\
\sin \alpha  & \cos \alpha  & 0 & 0 \\
0 & 0 & 1 & \sigma  \\
0 & 0 & 0 & 1
\end{pmatrix}
\end{equation*}
\begin{equation*}
\rho (b)=
\begin{pmatrix}
\cos \alpha  & -\sin \alpha  & 0 & 1-\cos \alpha  \\
\sin \alpha  & \cos \alpha  & 0 & -\sin \alpha  \\
0 & 0 & 1 & \sigma  \\
0 & 0 & 0 & 1
\end{pmatrix}
\end{equation*}
where $\alpha $ is the angle of rotation around the axis and $\sigma $ is
the shift.

The relation $aba=bab$ implies
\begin{eqnarray*}
\rho (aba)-\rho (bab) &=&
\begin{pmatrix}
0 & 0 & 0 & (1-2\cos \alpha )^{2}\sin \alpha  \\
0 & 0 & 0 & (1-2\cos \alpha )^{2}\sin \alpha  \\
0 & 0 & 0 & 0 \\
0 & 0 & 0 & 0
\end{pmatrix}
=
\begin{pmatrix}
0 & 0 & 0 & 0 \\
0 & 0 & 0 & 0 \\
0 & 0 & 0 & 0 \\
0 & 0 & 0 & 0
\end{pmatrix}
\\
&\Longrightarrow &\quad \cos \alpha =\frac{1}{2}\quad \Longrightarrow \quad
\alpha =\frac{2\pi }{6}\quad \Longrightarrow \quad x=\frac{\sqrt{3}}{2}
\end{eqnarray*}

If the parameter $\sigma $ is $0$, then $S$ is conjugate to the holonomy of
the Euclidean 2-dimensional orbifold $S_{236}$.

If the parameter $\sigma $ is different from $0,$ we assume up to similarity
that $\sigma =1/6.$ Then $S$ is conjugate to the Euclidean crystallographic
group $P6_{1}.$ To prove this, we compute some elements and their axes.
\begin{equation}
\mathbf{A}=\rho (a)=
\begin{pmatrix}
\frac{1}{2} & -\frac{\sqrt{3}}{2} & 0 & 0 \\
\frac{\sqrt{3}}{2} & \frac{1}{2} & 0 & 0 \\
0 & 0 & 1 & \frac{1}{6} \\
0 & 0 & 0 & 1
\end{pmatrix}
,\mathbf{B}=\rho (b)=
\begin{pmatrix}
\frac{1}{2} & -\frac{\sqrt{3}}{2} & 0 & \frac{3}{2} \\
\frac{\sqrt{3}}{2} & \frac{1}{2} & 0 & -\frac{\sqrt{3}}{2} \\
0 & 0 & 1 & \frac{1}{6} \\
0 & 0 & 0 & 1
\end{pmatrix}
\label{eroarob}
\end{equation}
The element $\mathbf{A}^{6}=\mathbf{B}^{6}$ is a translation by vector $
(0,0,1)$.
\begin{equation*}
\mathbf{ABA}=
\begin{pmatrix}
-1 & 0 & 0 & \frac{3}{2} \\
0 & -1 & 0 & \frac{\sqrt{3}}{2} \\
0 & 0 & -1 & \frac{1}{2} \\
0 & 0 & 0 & 1
\end{pmatrix}
,\quad (\mathbf{ABA})^{2}=
\begin{pmatrix}
1 & 0 & 0 & 0 \\
0 & 1 & 0 & 0 \\
0 & 0 & 1 & 1 \\
0 & 0 & 0 & 1
\end{pmatrix}
\end{equation*}
Then $\mathbf{ABA}$ is a rotation by $\pi $ with shift $(0,0,\frac{1}{2}).$
The axis of $\mathbf{ABA}$ is obtained by solving the equation
\begin{equation*}
\mathbf{ABA}
\begin{pmatrix}
x_{1} \\
x_{2} \\
x_{3} \\
1
\end{pmatrix}
-
\begin{pmatrix}
x_{1} \\
x_{2} \\
x_{3} \\
1
\end{pmatrix}
=
\begin{pmatrix}
0 \\
0 \\
\frac{1}{2} \\
0
\end{pmatrix}
\end{equation*}
Then $x_{1}=\frac{3}{4},$ $x_{2}=\frac{\sqrt{3}}{4}.$ The axis of $\mathbf{
ABA}$ is the line $\left( \frac{3}{4},\frac{\sqrt{3}}{4},t\right) $.
\begin{equation*}
\mathbf{AB}=
\begin{pmatrix}
-\frac{1}{2} & -\frac{\sqrt{3}}{2} & 0 & \frac{3}{2} \\
\frac{\sqrt{3}}{2} & -\frac{1}{2} & 0 & \frac{\sqrt{3}}{2} \\
0 & 0 & 1 & \frac{1}{3} \\
0 & 0 & 0 & 1
\end{pmatrix}
\end{equation*}
The isometry $\mathbf{AB}$ is a rotation by $\frac{2\pi }{3}$ with shift $
(0,0,\frac{1}{3}).$ The axis of $\mathbf{AB}$ is the line $\left( \frac{1}{2}
,\frac{\sqrt{3}}{2},t\right) $. These data correspond to the
crystallographic group $P6_{1}$ number 169 in the Tables \cite{Tables}.

The group $P6_{1}$ is also $\pi _{1}^{o}(\widehat{S}_{2_{1}3_{1}6_{1}})$,
the fundamental group of the 3-dimensional Euclidean orbifold $E^{3}/P6_{1}$
denoted by $\widehat{S}_{2_{1}3_{1}6_{1}}$, see Figure \ref{fp61}.

\begin{figure}[ht]
\epsfig{file=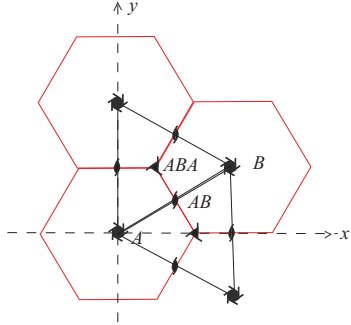,height=4.5cm} \caption{The crystallographic group $P6_{1}$.}\label{fp61}
\end{figure}
\end{proof}

\begin{theorem}
The homomorphism
\begin{equation*}
\begin{pmatrix}
\rho : & G(3/1) & \longrightarrow & P6_{1}\subset \mathcal{E}(E^{3}) \\
& a & \rightarrow & \mathbf{A} \\
& b & \rightarrow & \mathbf{B}
\end{pmatrix}
\end{equation*}
where $\mathbf{A}$ and $\mathbf{B}$ are given by \eqref{eroarob}, factors
through $\pi _{1}(M_{0}),$ where $M_{0}$ is the 3-manifold obtained from $
S^{3}$by 0-surgery in the trefoil knot.
\end{theorem}

\begin{proof}
The manifold $M_{0}$ is the result of pasting a solid torus $T$ to the
exterior of the trefoil knot $K$ such that the meridian of the torus $T$ is
mapped to the canonical longitude $l_{c}$ of $K$. The element of $G(3_{1})$
represented by $l_{c}$ is $a^{-4}baab$. Then
\begin{equation*}
\pi _{1}(M_{0})=\left\vert a,b;aba=bab,a^{-4}baab\right\vert
\end{equation*}
The image $\rho (l_{c})=\mathbf{A}^{-4}\mathbf{BAAB}=I_{4\times 4},$
therefore the homomorphism $\rho $ factors through $\pi _{1}(M_{0})$: $\rho
= $ $\rho ^{\prime }\circ \eta $
\begin{equation*}
G(3/1)\overset{\eta }{\longrightarrow }\pi _{1}(M_{0})\overset{\rho ^{\prime
}}{\longrightarrow }P6_{1}
\end{equation*}
\end{proof}

\begin{corollary}
The exterior of the trefoil knot has a Euclidean structure whose completion
gives a complete Euclidean structure in $M_{0}$.
\end{corollary}

The following Theorem gives more information on this manifold $M_{0}$.

\begin{theorem}
The manifold $M_{0}$ is the spherical tangent bundle of the 2-dimensional
Euclidean orbifold $S_{2,3,6}$.
\end{theorem}

\begin{proof}
The sphere $S^{3}$ has the Seifert manifold structure $(Ooo|0;(3,-1),(2,1))$
, with the trefoil knot $K$ as general fibre. The result of 0-surgery in a
general fibre produces a Seifert manifold $(Ooo|0;(3,-1),(2,1),(\alpha
,\beta ))$, where the pair $(\alpha ,\beta )$ can be easily computed as
follows. Let $Q$, $Q_{1}$ and $Q_{2}$ be simple closed curves in $S^{2}$, the
base of the Seifert fibration $(Ooo|0;(3,-1),(2,1))$, which are meridians of
a general fibre $H$, the exceptional fibre $(3,-1)$, and the exceptional
fibre $(2,1)$, respectively. Then, the first homology group of $(S^{3}\setminus K)$ has the following presentation
\begin{equation*}
H_{1}(S^{3}\setminus
K)=|Q,Q_{1},Q_{2},H;3Q_{1}-H=0,2Q_{2}+H=0,Q+Q_{1}+Q_{2}=0|
\end{equation*}
Let $l$ be the canonical longitude of $K=H$, then in $M_{0}$ we have that
\begin{equation*}
l=\alpha Q+\beta H=\alpha (-Q_{1}-Q_{2})+\beta 3Q_{1})=Q_{1}(3\beta -\alpha
)-\alpha Q_{2}=0
\end{equation*}
This implies that
\begin{equation*}
(3\beta -\alpha )Q_{1}=\alpha Q_{2}\quad \Rightarrow \quad -\alpha
Q_{1}=\alpha Q_{2}+2\beta Q_{2}
\end{equation*}
But also $2Q_{2}=-3Q_{1}$. Therefore
\begin{equation*}
\frac{3\beta -\alpha }{\alpha }=-\frac{3}{2}\quad \Rightarrow \quad 6\beta
-2\alpha =-3\alpha \quad \Rightarrow \quad 6\beta =-\alpha \quad \Rightarrow
\quad \alpha =6,\,\beta =-1.
\end{equation*}
Using Seifert signature equivalence, we have that
\begin{equation*}
M_{0}=(Ooo|0;(3,-1),(2,1),(6,-1))=(Ooo|1;(3,-1),(2,-1),(6,-1))
\end{equation*}
which is orientation reversing equivalent to
\begin{equation*}
(Ooo|1;(3,1),(2,1),(6,1))=ST(S_{2,3,6})
\end{equation*}
the spherical tangent of the Euclidean orbifold $S_{6,3,2}$ since
\begin{equation*}
\chi =-1+\frac{1}{2}+\frac{1}{3}+\frac{1}{6}=0
\end{equation*}
\end{proof}

\begin{remark}
The manifold $M_{0}$ is a torus bundle over $S^{1}$ with periodic monodromy
of order 6 (\cite{Z1965}).
\end{remark}

Theorem \ref{tposiblesH} can be used to identify all the 3-dimensional
orientation preserving Euclidean crystallographic groups generated by one
element of order two and one other of order three.

\begin{theorem}
\label{tposibleseucli}The only 3-dimensional orientation preserving
Euclidean crystallographic groups generated by one element of order two and
one other of order three, up to similarity, are $I2_{1}3$, $P4_{1}32(P4_{3}32)$
and $P6_{1}$.
\end{theorem}

\begin{proof}
Assume $S$ is conjugate to the subgroup $\rho _{x}(G(3/1)$ generated by $
M_{x}(a)$ and $M_{x}(b)$, $x=\cos \frac{\alpha }{2}\in (0,\sqrt{3}/2)$,
where $M_{x}(a)$ and $M_{x}(b)$ are given by \eqref{emxa} and \eqref{emxb}.
A necessary condition for $S$ be Euclidean crystallographic group is $\alpha
\in \left\{ \frac{2\pi }{2},\frac{2\pi }{3},\frac{2\pi }{4},\frac{2\pi }{6}
\right\} \Longleftrightarrow x\in \left\{ 0,\frac{1}{2},\frac{\sqrt{2}}{2},
\frac{\sqrt{3}}{2}\right\} $.

\begin{enumerate}
\item [Case $x=0.$] This case is impossible. (Remark \ref{rro0}.)

\item [Case $x=\frac{1}{2}$.] Then
\begin{equation*}
\alpha =\frac{2\pi }{3},\quad \cos \omega =-\frac{1}{3},\quad \sigma =\frac{
\sqrt{3}}{2},\quad \delta =\frac{\sqrt{2}}{4}
\end{equation*}
We will prove that $\rho _{\frac{1}{2}}(G(3/1)$ is the Euclidean
crystallographic group $I2_{1}3$ (number 199 in Tables \cite{Tables}).

\begin{enumerate}
\item $\rho _{\frac{1}{2}}(G(3/1)\leqslant I2_{1}3$. To compare both groups,
we conjugate $\rho _{\frac{1}{2}}(G(3/1)$ by a similarity so that
\begin{equation*}
\alpha =\frac{2\pi }{3},\quad \cos \omega =-\frac{1}{3},\quad \sigma =\frac{
\sqrt{3}}{6},\quad \delta =\frac{\sqrt{2}}{12}
\end{equation*}
The axes $\rho _{\frac{1}{2}}(a)$ and $\rho _{\frac{1}{2}}(b)$ are depicted
in Figure \ref{fabI23}. It can be checked that the two axes in Figure 
\ref{fabI23} have distance $\frac{\sqrt{2}}{12}=\delta $ and the angle $
\omega $ is
\begin{equation*}
\cos \omega =2\cos ^{2}\frac{\omega }{2}-1=2\left( \frac{1}{\sqrt{3}}\right)
^{2}-1=-\frac{1}{3}
\end{equation*}
The shift $\sigma =\frac{\sqrt{3}}{6}$ is the length of the diagonal of the
cube with edge length $\frac{1}{6}$.Therefore $\rho _{\frac{1}{2}}(a^{-1})$
and $\rho _{\frac{1}{2}}(b)$ are the screws axes in $I2_{1}3$ (Figure \ref
{fi213} $)$ denoted by $\mathbf{A}^{-1}$ and $\mathbf{B}$.

\begin{figure}[ht]
\epsfig{file=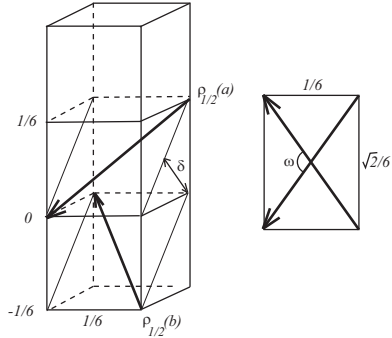,height=4.5cm} \caption{The axes of $\rho _{\frac{1}{2}}(a)$
and $\rho _{\frac{1}{2}}(b)$.}\label{fabI23}
\end{figure}

\begin{figure}[ht]
\epsfig{file=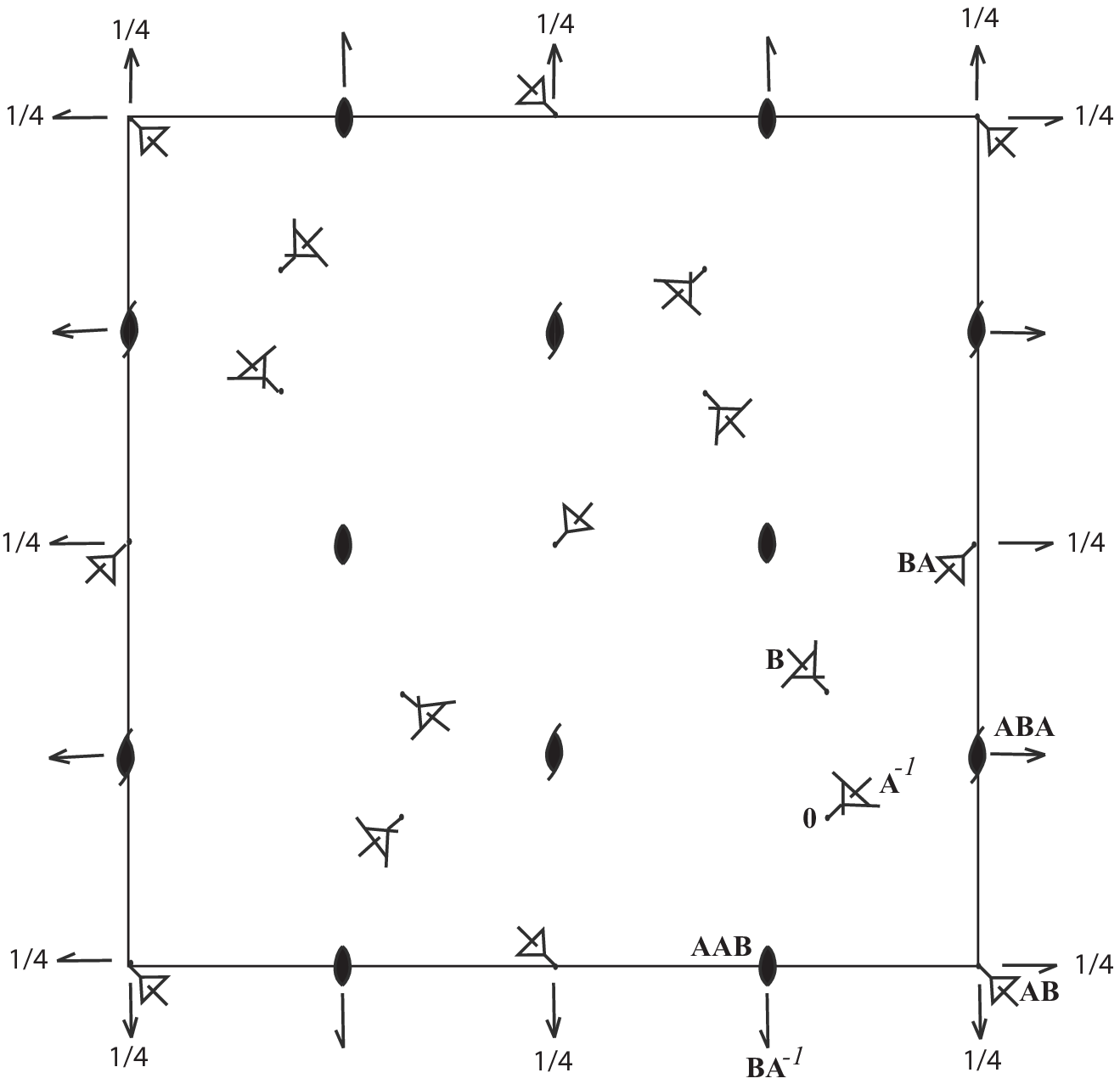,height=6.5cm} \caption{The crystallographic group $I2_{1}3$.}\label{fi213}
\end{figure}

\item $\rho _{\frac{1}{2}}(G(3/1)=I2_{1}3$. The group $I2_{1}3$ satisfies
the following short exact sequence
\begin{equation*}
0\longrightarrow T^{3}\longrightarrow I2_{1}3\overset{p}{\longrightarrow }
S_{233}\longrightarrow 1
\end{equation*}
where $T^{3}$ is the translational subgroup of a cube with edge length $1$
as fundamental domain and $S_{233}$, the linear quotient, is the holonomy of
the 2-dimensional spherical orbifold denoted also by $S_{233}$. It is clear
that $p(\rho _{\frac{1}{2}}(a^{-1}))=\rho _{\frac{1}{2}}^{\prime }(a^{-1})$
and $p(\rho _{\frac{1}{2}}(b))=\rho _{\frac{1}{2}}^{\prime }(b)$ generate $
S_{233}$ by Theorem \ref{tgenso3}. Therefore it suffices to prove that $
T^{3}\leqslant \rho _{\frac{1}{2}}(G(3/1)$ and to find enough elements: The
element $\mathbf{A}^{3}$ is a translation with vector $(\frac{1}{2},\frac{1}{
2},\frac{1}{2})$. Figure \ref{fi213} shows the elements $\mathbf{AB}$, $
\mathbf{BA}$ (3-fold rotations), $\mathbf{AAB}$ (2-fold rotation), $\mathbf{
ABA}$, $\mathbf{BA}^{-1}$ (2-screw rotations).
\end{enumerate}

\item [Case $x=\frac{\sqrt{2}}{2}$.] Then
\begin{equation*}
\alpha =\frac{2\pi }{4},\quad \cos \omega =0,\quad \sigma =\frac{1}{4},\quad
\delta =\frac{1}{4}
\end{equation*}
The group $P4_{1}32$ is depicted in Figure \ref{fi4132}. Up to similarity we assume that
the axes $\rho _{\frac{\sqrt{2}}{2}}(a)$ and $\rho _{\frac{\sqrt{2}}{2}}(b)$
are the ones depicted in Figure \ref{fabP432} , so they are the $4_{1}$
axes denoted by $\mathbf{A}$ and $\mathbf{B}$ in Figure \ref{fi4132}.

\begin{figure}[ht]
\epsfig{file=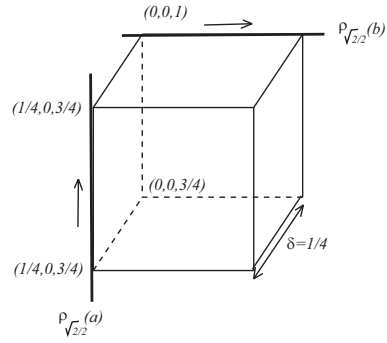,height=4.5cm} \caption{The axes of $\rho _{\frac{
\sqrt{2}}{2}}(a)$ and $\rho _{\frac{\sqrt{2}}{2}}(b)$.}\label{fabP432}
\end{figure}

\begin{figure}[ht]
\epsfig{file=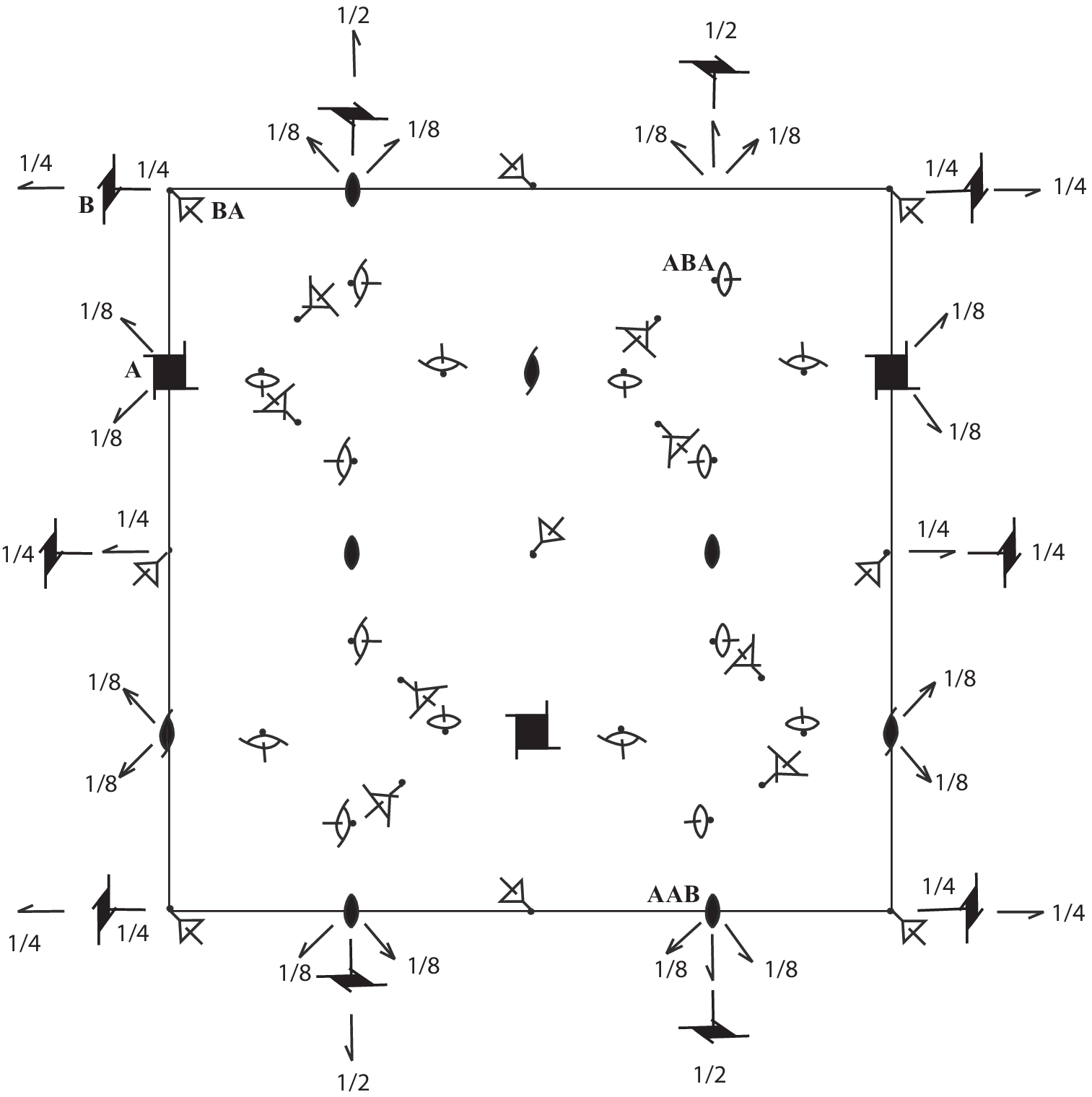,height=8.5cm} \caption{The crystallographic group $I4_{1}32$.}\label{fi4132}
\end{figure}

In affine notation
\begin{equation}
\mathbf{A}=\rho _{\frac{\sqrt{2}}{2}}(a)=
\begin{pmatrix}
0 & -1 & 0 & \frac{1}{4} \\
1 & 0 & 0 & -\frac{1}{4} \\
0 & 0 & 1 & \frac{1}{4} \\
0 & 0 & 0 & 1
\end{pmatrix}
,\quad \mathbf{B}=\rho _{\frac{\sqrt{2}}{2}}(b)=
\begin{pmatrix}
0 & 0 & 1 & -\frac{1}{4} \\
0 & 1 & 0 & \frac{1}{4} \\
-1 & 0 & 0 & \frac{1}{4} \\
0 & 0 & 0 & 1
\end{pmatrix}
\end{equation}

This proves that $\rho _{\frac{\sqrt{2}}{2}}(G(3/1)\leqslant P4_{1}32$,
because $\rho _{\frac{\sqrt{2}}{2}}(G(3/1)$ is generated by $\mathbf{A}$ and
$\mathbf{B}$, both elements of $P4_{1}32$. To prove the equality, we can
compute the others elements of $P4_{1}32$.  For instance, the element
\begin{equation*}
\mathbf{BA}=
\begin{pmatrix}
0 & 0 & 1 & 0 \\
1 & 0 & 0 & 0 \\
0 & 1 & 0 & 0 \\
0 & 0 & 0 & 1
\end{pmatrix}
\end{equation*}
is a rotation by $2\pi /3$ with axis $(t,t,t)$. The element
\begin{equation*}
\mathbf{ABA}=
\begin{pmatrix}
-1 & 0 & 0 & \frac{1}{4} \\
0 & 0 & 1 & -\frac{1}{4} \\
0 & 1 & 0 & \frac{1}{4} \\
0 & 0 & 0 & 1
\end{pmatrix}
\end{equation*}
is a rotation by $\pi $ with axis $(1/8,t,t+1/4)$. Observe that $\mathbf{A}
^{4}$ and $\mathbf{B}^{4}$ are translations. The subgroup generated by $
\mathbf{A}^{4}$ and $\mathbf{B}^{4}$ has a cube as fundamental domain, with edge length  one.

\item [Case $x=\frac{\sqrt{3}}{2}.$] Theorem \ref{tposiblesH} 2.c) shows that $
S $ is conjugate to $P6_{1}$.
\end{enumerate}
\end{proof}

\begin{remark}
The space $E^{3}/I2_{1}3$ is the Euclidean orbifold $Q_{1}$ with underlying
space $S^{3}$ and singular set the rational link $10/3$ whose two components
have isotropies or orders $2$ and $3$. See Figure \ref{forbeuclideas}.
\end{remark}

\begin{remark}
The space $E^{3}/P4_{1}32$ is the Euclidean orbifold $Q_{2}$ with underlying
space $S^{3}$ and singular set the graph depicted in Figure \ref
{forbeuclideas}. This orbifold is 2-fold covered by the Euclidean orbifold $
E^{3}/P2_{1}3$ with underlying space $S^{3}$ and singular set the Figure
Eight knot with isotropy of order $3$.

\begin{figure}[ht]
\epsfig{file=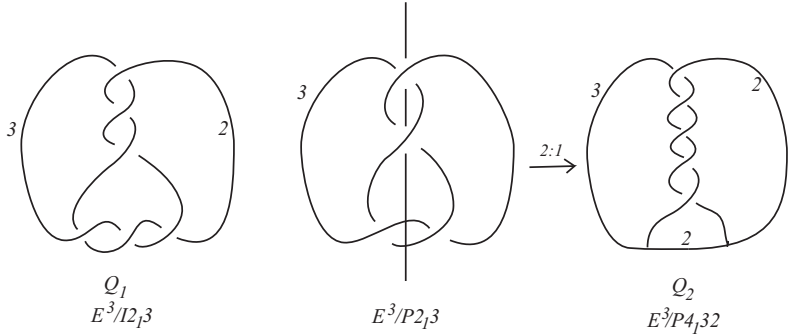,height=4.5cm} \caption{Euclidean orbifolds.}\label{forbeuclideas}
\end{figure}

\end{remark}

\subsection{Case 2: $|x|=\frac{\protect\sqrt{3}}{2}$}
The almost-irreducible representation $\rho _{\pm \frac{\sqrt{3}}{2}
}:G(3/1)\longrightarrow SL(2,\mathbb{C})$ given by \eqref{ecaso2} do not
have a proper affine deformation, because the polynomial \eqref{epoly2} in
this case gives $s=0$. But, by Theorem \ref{tposiblesH}, there exists an
affine deformation of a reducible representation $\rho _{\frac{\sqrt{3}}{2}
}:G(3/1)\longrightarrow SO(3)$ , whose image is the crystallographic group $
P6_{1}$.

\subsection{The quaternion algebra $\left( \frac{-1,1}{\mathbb{R}}\right) $.}
Next we analyze the affine deformations of representations corresponding to
points in the character variety belonging to Cases 3, 4 and 5: $|x|>\frac{
\sqrt{3}}{2}$, where the quaternion algebra is $M(2,\mathbb{R})=\left( \frac{
-1,1}{\mathbb{R}}\right) $. Therefore $U_{1}=SL(2,\mathbb{R})$, $
H_{0}=E^{1,2}$, the Minkowsi space, and $A(H)=\mathcal{L}(E^{1,2})$.

The meaning of the parameters depends on the sign of $x-1$.
\begin{enumerate}
\item [Case 3]: $\frac{\sqrt{3}}{2}<x<1\Longrightarrow x-1<0$, then $A^{-}$ is
a vector inside the nullcone (a time-like vector) and $A$ acts as a right
spherical rotation around the axis $A^{-}$ with angle $\alpha $, where $
x=\cos \frac{\alpha }{2}$. Let $d$ be the hyperbolic distance between the
projection of $A^{-}$ and $B^{-}$ on the hyperbolic plane (pure unit
quaternions in the upper component), then $\cosh d=\frac{y}{1-x^{2}}$.
\begin{eqnarray}
x &=&A^{+}=B^{+}=\cos \frac{\alpha }{2}  \notag \\
u &=&N(A^{-})=N(B^{-})=1-x^{2}=\sin ^{2}\frac{\alpha }{2}  \notag \\
y &=&u\cosh d
\end{eqnarray}

\item [Case 4]: $x=1,$ then $A^{-}$ belongs to the nullcone ( a light-like
vector) and $A$ acts as a parabolic transformation around the axis $A^{-}$.

\item [Case 5]: $1<x\Leftrightarrow $ $x-1>0$, then $A^{-}$ is a vector
outside the nullcone (a space-like vector) and $A$ acts as a right
hyperbolic rotation around the axis $A^{-}$ with distance $\partial $, where
$x=\cosh \frac{\partial }{2}$. The meaning of the parameter $y$ depends on
the sign of $y^{2}-(x^{2}-1)^{2}$.

\begin{enumerate}
\item If $y^{2}-(x^{2}-1)^{2}>0$, $y^{2}=(x^{2}-1)^{2}\cosh ^{2}d$, where $d$
is the distance between the polars of $A^{-}$ and $B^{-}$.

\item If $y^{2}-(x^{2}-1)^{2}=0$ then $y^{2}=(x^{2}-1)^{2}$.

\item If $y^{2}-(x^{2}-1)^{2}<0$, $y^{2}=(x^{2}-1)^{2}\cos ^{2}\theta $,
where $\theta $ is the angle between the polars of $A^{-}$ and $B^{-}$.

Therefore
\begin{eqnarray}
x &=&A^{+}=B^{+}=\cosh \frac{\partial }{2}  \notag \\
y^{2} &>&(x^{2}-1)^{2}\Longrightarrow y^{2}=(x^{2}-1)^{2}\cosh ^{2}d \\
y^{2} &<&(x^{2}-1)^{2}\Longrightarrow y^{2}=(x^{2}-1)^{2}\cos ^{2}\theta
\end{eqnarray}

The shift $\sigma $ of the element $(sA^{-},A)$ and the distance $\delta $
between the axes of $(sA^{-},A)$ and $(sB^{-}+(A^{-}B^{-})^{-},B)$ are as in
the case 1, that is
\begin{equation*}
\sigma =s\sqrt{u}=s\sqrt{1-x^{2}}.
\end{equation*}
\begin{equation}
N((A^{-}B^{-})^{-})=-y^{2}+u^{2}\quad \Longrightarrow \quad \delta =\frac{
\sqrt{u^{2}-y^{2}}}{2}
\end{equation}
\end{enumerate}
\end{enumerate}

\begin{theorem}
\label{tgenlorentz} For each $x\in (\sqrt{3}/2,\infty ),$ there exists a
representation $\rho _{x}:G(3/1)\longrightarrow \mathcal{L}(E^{1,2})$ unique
up to conjugation in $\mathcal{L}(E^{1,2})$ such that
\begin{equation*}
\begin{array}{l}
\rho _{x}(a)=(sA^{-},A) \\
\rho _{x}(b)=((sB^{-}+(A^{-}B^{-})^{-},B)
\end{array}
\end{equation*}
where the values of $A,B\in SL(2,\mathbb{R})$ are the following

\begin{enumerate}
\item For $\frac{\sqrt{3}}{2}<x<1$.
\begin{equation*}
\begin{array}{c}
A=x+\sqrt{1-x^{2}}i,\quad \sqrt{1-x^{2}}>0 \\
B=x+\frac{2x^{2}-1}{2\sqrt{1-x^{2}}}i+\frac{1}{2}\sqrt{\frac{4x^{2}-3}{
1-x^{2}}}j
\end{array}
\end{equation*}

\item For $x=1\Longrightarrow $ $y=\frac{1}{2}$
\begin{equation*}
\begin{array}{c}
A=1+i+j \\
B=1+\frac{1}{4}(i-j)
\end{array}
\end{equation*}

\item For $x>1$, $y=\frac{2x^{2}-1}{2}>0$
\begin{equation*}
\begin{array}{c}
\widehat{\rho }_{x}(a)=A=x+\sqrt{x^{2}-1}j,\quad \sqrt{x^{2}-1}>0 \\
\widehat{\rho }_{x}(b)=B=x-\frac{1}{2}\sqrt{\frac{4x^{2}-3}{x^{2}-1}}i-\frac{
2x^{2}-1}{2\sqrt{x^{2}-1}}j
\end{array}
\end{equation*}
\end{enumerate}

In affine linear notation, where $\left\{ X,Y,Z\right\} $ is the coordinate
system associated to the basis $\left\{ -ij,j,i\right\} $
\begin{equation*}
\begin{array}{l}
\rho _{x}(a)=M_{x}(a)\left(
\begin{array}{c}
X \\
Y \\
Z \\
1
\end{array}
\right) =\left(
\begin{array}{c}
X^{\prime } \\
Y^{\prime } \\
Z^{\prime } \\
1
\end{array}
\right) \\
\rho _{x}(b)=M_{x}(b)\left(
\begin{array}{c}
X \\
Y \\
Z \\
1
\end{array}
\right) =\left(
\begin{array}{c}
X^{\prime } \\
Y^{\prime } \\
Z^{\prime } \\
1
\end{array}
\right)
\end{array}
\end{equation*}
where,

\begin{enumerate}
\item For $\frac{\sqrt{3}}{2}<x<1$
\begin{equation}
M_{x}(a)=\left(
\begin{array}{cccc}
2x^{2}-1 & -2x\sqrt{1-x^{2}} & 0 & 0 \\
2x\sqrt{1-x^{2}} & 2x^{2}-1 & 0 & 0 \\
0 & 0 & 1 & \frac{\left( 3-4x^{2}\right) \sqrt{1-x^{2}}}{4x} \\
0 & 0 & 0 & 1
\end{array}
\right)  \label{emacase3}
\end{equation}
\begin{equation}
M_{x}(b)=
\begin{pmatrix}
2x^{2}-1 & \frac{x-2x^{3}}{\sqrt{1-x^{2}}} & x\sqrt{\frac{3-4x^{2}}{x^{2}-1}}
& -\frac{1}{2}\sqrt{4x^{2}-3} \\
\frac{x(2x^{2}-1)}{\sqrt{1-x^{2}}} & \frac{1+2x^{2}-4x^{4}}{2-2x^{2}} &
\frac{\sqrt{4x^{2}-3}(2x^{2}-1)}{2(1-x^{2})} & \frac{\sqrt{(4x^{2}-3)^{3}}}{
8x\sqrt{1-x^{2}}} \\
x\sqrt{\frac{3-4x^{2}}{x^{2}-1}} & \frac{\sqrt{4x^{2}-3}(2x^{2}-1)}{
2(1-x^{2})} & \frac{1-2x^{2}}{2x^{2}-2} & \frac{-3+10x^{2}-9x^{4}}{8x\sqrt{
1-x^{2}}} \\
0 & 0 & 0 & 1
\end{pmatrix}
\label{embcase3}
\end{equation}

\item For $x=1\Longrightarrow $ $y=\frac{1}{2}$
\begin{equation}
M_{x}(a)=\left(
\begin{array}{cccc}
1 & -2 & 2 & 0 \\
2 & -1 & 2 & -\frac{1}{4} \\
2 & -2 & 3 & -\frac{1}{4} \\
0 & 0 & 0 & 1
\end{array}
\right)  \label{emacase4}
\end{equation}
\begin{equation}
M_{x}(b)=
\begin{pmatrix}
1 & -\frac{1}{2} & -\frac{1}{2} & \frac{1}{2} \\
\frac{1}{2} & \frac{7}{8} & -\frac{1}{8} & \frac{1}{16} \\
-\frac{1}{2} & \frac{1}{8} & \frac{9}{8} & -\frac{1}{16} \\
0 & 0 & 0 & 1
\end{pmatrix}
\label{embcase4}
\end{equation}

\item For $x>1$, $y=\frac{2x^{2}-1}{2}>0$
\begin{equation}
M_{x}(a)=\left(
\begin{array}{cccc}
2x^{2}-1 & 0 & 2x\sqrt{x^{2}-1} & 0 \\
0 & 1 & 0 & \frac{\left( 3-4x^{2}\right) \sqrt{x^{2}-1}}{4x} \\
2x\sqrt{x^{2}-1} & 0 & 2x^{2}-1 & 0 \\
0 & 0 & 0 & 1
\end{array}
\right)  \label{emacase5}
\end{equation}
\begin{equation}
M_{x}(b)=
\begin{pmatrix}
2x^{2}-1 & x\sqrt{\frac{-3+4x^{2}}{x^{2}-1}} & -\frac{x(2x^{2}-1)}{\sqrt{
x^{2}-1}} & -\frac{1}{2}\sqrt{4x^{2}-3} \\
-x\sqrt{\frac{-3+4x^{2}}{x^{2}-1}} & \frac{1-2x^{2}}{2x^{2}-2} & \frac{\sqrt{
-3+4x^{2}}(2x^{2}-1)}{2(x^{2}-1)} & \frac{3-10x^{2}+8x^{4}}{8x\sqrt{x^{2}-1}}
\\
-\frac{x(2x^{2}-1)}{\sqrt{x^{2}-1}} & \frac{\sqrt{-3+4x^{2}}(2x^{2}-1)}{
2(x^{2}-1)} & \frac{1+2x^{2}-4x^{4}}{2-2x^{2}} & -\frac{\sqrt{(-3+4x^{2})^{3}
}}{8x\sqrt{x^{2}-1}} \\
0 & 0 & 0 & 1
\end{pmatrix}
\label{mbcase5}
\end{equation}
The distance $\delta $ and the angle $\omega $ between the axis of $\rho
_{x}(a)$ and $\rho _{x}(b)$ are given by
\begin{eqnarray*}
\delta &=&\frac{\sqrt{3-4x^{2}}}{4} \\
\cos \omega &=&\frac{2x^{2}-1}{2-2x^{2}}
\end{eqnarray*}
The shift $\sigma $ is
\begin{equation*}
\sigma =(\frac{3}{4x}-x)\sqrt{1-x^{2}}
\end{equation*}
\end{enumerate}

\begin{proof}
Theorem \ref{teorema4} gives the values of $A$ and $B$ as a function of $x$
and $y$ and also the values of $(A^{-}B^{-})^{-}$ in each case:
\begin{eqnarray*}
\frac{\sqrt{3}}{2} &<&x<1\Longrightarrow (A^{-}B^{-})^{-}=\sqrt{
y^{2}-(1-x^{2})^{2}}ij \\
x &=&1\Longrightarrow (A^{-}B^{-})^{-}=-yij \\
x &>&1\Longrightarrow (A^{-}B^{-})^{-}=\sqrt{y^{2}-(x^{2}-1)^{2}}ij
\end{eqnarray*}
From the polynomial defining the character variety \eqref{epoly1}
\begin{equation*}
2y-(2x^{2}-1)=0
\end{equation*}
we obtain the value of $y$
\begin{equation*}
y=\frac{2x^{2}-1}{2}
\end{equation*}
From the polynomial \eqref{epoly2} we obtain the value of $s$
\begin{equation*}
s=\frac{3-4x^{2}}{4x}
\end{equation*}
\cite[Th.5]{HLM2009} defines the representation $\rho
_{x}:G(3/1)\longrightarrow \mathcal{L}(E^{1,2})$ unique up to conjugation in
$\mathcal{L}(E^{1,2})$ such that
\begin{equation*}
\begin{array}{l}
\rho _{x}(a)=(sA^{-},A) \\
\rho _{x}(b)=((sB^{-}+(A^{-}B^{-})^{-},B)
\end{array}
\end{equation*}
\end{proof}
\end{theorem}

As in the Euclidean case we are interested in representations in $\mathcal{L}
(E^{1,2})$ whose image is an \emph{affine crystallographic group}. This
concept is defined by Fried and Goldman \cite{FG1983} as a subgroup of the affine group $
A(H)$ acting properly discontinuously and with compact orbit space. An
affine crystallographic group is the fundamental group of a flat affine
manifold. The following theorem analyzes some examples in case 1 $(x<1)$ of
Theorem \ref{tgenlorentz}, where generators go to elements with linear part $
A$ and $B$ such that $A^{-}$ and $B^{-}$ are timelike vectors.

\begin{theorem}
The image of the representation $\rho _{x}:G(3/1)\longrightarrow \mathcal{L}
(E^{1,2})$, where $x=\cos \frac{2\pi }{2n}$, $n\geq 7$, is not a properly
discontinuous subgroup of $\mathcal{L}(E^{1,2})$.
\end{theorem}

\begin{proof}
The linear quotient of $\rho _{x}(G(3/1))$ is the group $S_{23n}\subset
SO(1,2)$. The group $S_{23n}$, been the holonomy group of the 2-dimensional
hyperbolic orbifold denoted by the same name $S_{23n}$, is a cocompact
group, therefore Mess's Theorem (\cite{Mess1990} ,\cite{GM2000}) says that
$Im(\rho _{x})$ is not a properly discontinuous subgroup of $\mathcal{
L}(E^{1,2})$.
\end{proof}

To analyze the discrete condition of representations in case 3 $(x>1)$ of
Theorem \ref{tgenlorentz}, we can use the Margulis invariant
(\cite{Mar1983}, \cite{Mar1984}, \cite{GM2000}).
Recall that an element of $O(1,2)$ is hyperbolic if it has three distinct
real eigenvalues. A subgroup $G\subset $ $O(1,2)$ is purely hyperbolic if
every element is hyperbolic.

\begin{theorem}
Every image of $\rho _{x}:G(3/1)\longrightarrow \mathcal{L}(E^{1,2})$, where
$x>1$ contains an affine deformation of a purely hyperbolic subgroup of
finite index.
\end{theorem}

\begin{proof}
Theorem \ref{tsubgrupolibre} shows that the image of the linear quotient of $
\rho _{x}$, $\rho _{x}^{\prime }=\pi _{2}\circ \rho
_{x}:G(3/1)\longrightarrow SO^{0}(1,2)$, $1<x<\infty $, is an index 6 free
subgroup  generated by $\{\rho _{x}^{\prime }(b^{-2}),$ $\rho _{x}^{\prime
}(a^{2})\}$. All the elements of this subgroup are hyperbolic
transformations.

We conclude that the subgroup of $\rho _{x}(G(3/1))$ generated by $\{\rho
_{x}(b^{-2}),$ $\rho _{x}(a^{2})\}$ is an affine deformation of the purely
hyperbolic free subgroup of rank 2 generated by $\{B^{-2},A^{2}\}$.
\end{proof}

Margulis defines an invariant $\alpha _{\phi }:G\longrightarrow \mathbb{R}$
of an affine deformation $\phi $ of a purely hyperbolic subgroup $G\subset
SO^{0}(1,2)$ as follows. Every element $g\in G$ has three distinct positive
real eigenvalues $\lambda (g)<1<\lambda (g)^{-1}$. Choose an eigenvector $
x^{-}(g)$ for $\lambda (g)$ and an eigenvector $x^{+}(g)$ for $\lambda
(g)^{-1}$, both in the same component $\mathcal{N}_{+}$ of the complement of
$0$ in the nullcone. Consider the unique eigenvector $x^{0}(g)$ for $g$ with
eigenvalue $1$ such that $|x^{0}(g)|=-1$ and $\{x^{-}(g),x^{+}(g),x^{0}(g)\}$
is a positively oriented basis of $E^{1,2}$.

If $\phi $ is a hyperbolic deformation of $G$, then $\alpha _{\phi }$ is
defined as
\begin{equation*}
\begin{array}{cccc}
\alpha _{\phi }: & G & \longrightarrow & \mathbb{R} \\
& g & \rightarrow & Q(x^{0}(g),\phi (g)(x)-x)
\end{array}
\end{equation*}
for any $x\in E^{1,2}$, where $Q(-,-)$ is the bilinear quadratic form
defining the Minkowski metric. It has been proven (\cite{DG2001}) that $\alpha $
is a complete invariant of conjugacy class of the affine deformation. The
following theorem of Margulis can be used to check the
proper condition of an affine deformation.

\begin{theorem}[Margulis] Let $G$ be a purely hyperbolic subgroup of $SO^{0}(1,2)$, and $
\phi :G\longrightarrow \mathcal{L}(E^{1,2})$ an affine deformation. If there
exist $g_{1},g_{2}\in G$ such that $\alpha _{\phi }(g_{1})>0>\alpha _{\phi
}(g_{2})$, then $\phi $ is not proper.
\end{theorem}

\begin{theorem}
The image of the representation $\rho _{x}:G(3/1)\longrightarrow \mathcal{L}
(E^{1,2})$, where $x>1$, is not a properly discontinuous subgroup of $
\mathcal{L}(E^{1,2})$.

\begin{proof}
We compute the Margulis invariants of the elements $g_{2}=\widehat{\rho }
_{x}^{\prime }(a^{2})$ and $g_{3}=\widehat{\rho }_{x}^{\prime }(a^{2}b^{2})$
in the purely hyperbolic group $\pi _{1}^{o}(O_{2\delta ,2\delta ,2\delta
})\subset \widehat{\rho }_{x}^{\prime }(G(3/1))$. Observe that the element
\begin{equation*}
g_{2}=\widehat{\rho }_{x}^{\prime }(a^{2})=
\begin{pmatrix}
1-8x^{2}+8x^{4} & 0 & 4x\sqrt{x^{2}-1}(2x^{2}-1) \\
0 & 1 & 0 \\
4x\sqrt{x^{2}-1}(2x^{2}-1) & 0 & 1-8x^{2}+8x^{4}
\end{pmatrix}
\end{equation*}
has the following eigenvalues
\begin{eqnarray*}
\lambda (g_{2}) &=&1-8x^{2}+8x^{4}-\sqrt{(1-8x^{2}+8x^{4})^{2}-1}<1 \\
1 &<&\left( \lambda (g_{2})\right) ^{-1}=1-8x^{2}+8x^{4}+\sqrt{
(1-8x^{2}+8x^{4})^{2}-1}.
\end{eqnarray*}
We choose the eigenvector $x^{-}(g_{2})=\{-1,0,1\},$ $x^{+}(g_{2})=\{1,0,1\}.
$ Therefore the vector $x^{0}(g_{2})=\{0,1,0\}$ is the unique eigenvector
such that $|x^{0}(g)|=-1$ and $\{x^{-}(g),x^{+}(g),x^{0}(g)\}$ is a
positively oriented basis of $E^{1,2}$. The Margulis invariant $\alpha
_{\phi _{x}}(g_{2})=Q(x^{0}(g),\rho _{x}(a^{2})(x)-x)$ for any $x\in E^{1,2}$
is
\begin{equation*}
\alpha _{\phi _{x}}(g_{2})=(0,1,0)
\begin{pmatrix}
-1 & 0 & 0 \\
0 & -1 & 0 \\
0 & 0 & 1
\end{pmatrix}
\begin{pmatrix}
s_{1} \\
s_{2} \\
s_{3}
\end{pmatrix}
=-s_{2}
\end{equation*}
where $\{s_{1},s_{2},s_{3}\}$ are the $\{X,Y,Z\}-$coordinates of $\rho
_{x}(a^{2})(x)-x$ for $x=\{t_{1},t_{2},t_{3}\}$. The computation gives
\begin{equation*}
s_{2}=\frac{(3-4x^{2})\sqrt{x^{2}-1}}{2x}.
\end{equation*}
Therefore
\begin{equation*}
\alpha _{\phi }(g_{2})=-\frac{(3-4x^{2})\sqrt{x^{2}-1}}{2x}.
\end{equation*}
This value is only zero for $x=\pm \frac{\sqrt{3}}{2},\pm 1$, then it has
the same sign for all $x>1.$ For $x=2$, it is $\frac{13\sqrt{3}}{4}>0.$ Then
\begin{equation*}
\alpha _{\phi _{x}}(g_{2})>0,\quad x>1.
\end{equation*}
We have used the computer program Mathematica to do an analogous, but much
more complicate computation for $g_{3}=\widehat{\rho }_{x}^{\prime
}(a^{2}b^{2})$. We found
\begin{multline*}
\alpha _{\phi _{x}}(g_{3})=\\
\frac{3+x^{2}(-1-22x^{2}+36x^{4}-16x^{6}-8x\sqrt{
x^{2}-1}(9+4x^{2}(-21+63x^{2}-76x^{4}+32x^{6})))}{8x\sqrt{(x^{2}-1)^{5}}}
\end{multline*}
which is never zero for $x>1$, and for $x=2$ it is equal to $-\frac{
143(15+7616\sqrt{3}}{144\sqrt{3}}<0$ . Therefore $\alpha _{\rho
_{x}}(g_{3})<0$, for all $x>1$. We conclude that $\alpha _{\rho
_{x}}(g_{2})>0>\alpha _{\rho _{x}}(g_{3})$ for all $x>1$. Then we apply the
above Margulis's Theorem to deduce that $\rho _{x}(G(3/1))$ contains a
subgroup with no proper action.
\end{proof}
\end{theorem}

For case 2 $(x=1)$ in Theorem \ref{tgenlorentz}, where generators go to
elements with linear part $A$ and $B$ such that $A^{-}$ and $B^{-}$ are null
vectors, we know by Theorem \ref{tsubgrupolibre} that the image of $\rho
_{1}:G(3_{1})\longrightarrow \mathcal{L}(E^{1,2})$, contains an affine
deformation of a free subgroup of index 6 generated by the two parabolic
elements $\{B^{-2},A^{2}\}$. To analyze the discrete condition of the
representation $\rho _{1}:G(3_{1})\longrightarrow \mathcal{L}(E^{1,2})$ one
could use the generalization of the Margulis invariant method obtained in \cite{CD2005}
 for subgroups
generated by two parabolic elements. But in this case it does not work
because the generalized Margulis invariant of the parabolics elements $A^{2}$
and $B^{-2}$ is $0,$ as is the Margulis invariant of hyperbolic elements $
A^{2n}B^{2}$. The reason is that each of the these elements has a line of
fixed points. Moreover this property for some element implies a non properly
discontinuously action. Therefore we can establish the following theorem.

\begin{theorem}
The image of the representation $\rho _{1}:G(3_{1})\longrightarrow \mathcal{L
}(E^{1,2})$, is not a properly discontinuous subgroup of $\mathcal{L}
(E^{1,2})$.
\end{theorem}

\begin{proof}
The elements
\begin{equation*}
\rho _{1}(a^{n})=\left(
\begin{array}{cccc}
1 & -2 & 2 & 0 \\
2 & -1 & 2 & -\frac{1}{4} \\
2 & -2 & 3 & -\frac{1}{4} \\
0 & 0 & 0 & 1
\end{array}
\right) ^{n}
\end{equation*}
fix each point in the line $(1/8,t,t)$, and the action in a neighborhood of
this line is not discontinuous.
\end{proof}

\begin{corollary}
\label{cor20}There is no affine crystallographic group in $\mathcal{L}
(E^{1,2})$ which is a quotient of $G(3_{1}).$
\end{corollary}

\begin{corollary}
There is no affine crystallographic group in $\mathcal{L}(E^{1,2})$
generated by two isometries $\mu $ and $\nu $ such that $\mu ^{2}=\nu ^{3}=1.
$
\end{corollary}

\begin{proof}
This is a consequence of Corollary \ref{cor20} and Corollary \ref{cor1}.
\end{proof}


\begin{thebibliography}{10}

\bibitem{Tables}
{\em International Tables for Crystallography.Volume A: Space-group symmetry}.
\newblock Edited by Th. Hahn. Springer, 2005.

\bibitem{CD2005}
Virginie Charette and Todd~A. Drumm.
\newblock The {M}argulis invariant for parabolic transformations.
\newblock {\em Proc. Amer. Math. Soc.}, 133(8):2439--2447 (electronic), 2005.

\bibitem{CF1963}
Richard~H. Crowell and Ralph~H. Fox.
\newblock {\em Introduction to knot theory}.
\newblock Based upon lectures given at Haverford College under the Philips
  Lecture Program. Ginn and Co., Boston, Mass., 1963.

\bibitem{DG2001}
Todd~A. Drumm and William~M. Goldman.
\newblock Isospectrality of flat {L}orentz 3-manifolds.
\newblock {\em J. Differential Geom.}, 58(3):457--465, 2001.

\bibitem{FG1983}
David Fried and William~M. Goldman.
\newblock Three-dimensional affine crystallographic groups.
\newblock {\em Adv. in Math.}, 47(1):1--49, 1983.

\bibitem{BZ}
G.Burde and H.~Zieschang.
\newblock {\em Knots}, volume~5 of {\em Studies in Mathematics}.
\newblock Walter de Gruyter, 1985.

\bibitem{GM2000}
William~M. Goldman and Gregory~A. Margulis.
\newblock Flat {L}orentz 3-manifolds and cocompact {F}uchsian groups.
\newblock In {\em Crystallographic groups and their generalizations
  ({K}ortrijk, 1999)}, volume 262 of {\em Contemp. Math.}, pages 135--145.
  Amer. Math. Soc., Providence, RI, 2000.

\bibitem{HLM2009}
H.M. Hilden, M.~T. Lozano, and J.~M. Montesinos-Amilibia.
\newblock On representations of 2-bridge knot groups in quaternion algebras.
\newblock {\em preprint. ArXiv: 1001.3546}, 2009.

\bibitem{Mar1983}
G.~A. Margulis.
\newblock Free completely discontinuous groups of affine transformations.
\newblock {\em Dokl. Akad. Nauk SSSR}, 272(4):785--788, 1983.

\bibitem{Mar1984}
G.~A. Margulis.
\newblock Complete affine locally flat manifolds with a free fundamental group.
\newblock {\em Zap. Nauchn. Sem. Leningrad. Otdel. Mat. Inst. Steklov. (LOMI)},
  134:190--205, 1984.
\newblock Automorphic functions and number theory, II.

\bibitem{Mess1990}
G~Mess.
\newblock Lorentz spacetimes of constant curvature.
\newblock {\em preprint}, 1990.

\bibitem{Z1965}
E.~C. Zeeman.
\newblock Twisting spun knots.
\newblock {\em Trans. Amer. Math. Soc.}, 115:471--495, 1965.

\end{thebibliography}
\end{document}